\documentclass[11pt,a4paper,reqno]{amsart}

\usepackage[a4paper,lmargin=2cm,rmargin=2cm,tmargin=4cm,bmargin=4cm]{geometry}

\usepackage[english]{babel}

\usepackage[centertags]{amsmath}
\usepackage{amsfonts}
\usepackage{amssymb}
\usepackage{amsthm}
\usepackage{dsfont}

\usepackage[dvipsnames]{xcolor}

\usepackage[normalem]{ulem}
\usepackage{bbm}

\usepackage{hyperref}
	\hypersetup{breaklinks=true,colorlinks=true,
linkcolor=MidnightBlue,citecolor=MidnightBlue,
urlcolor=MidnightBlue}
\usepackage{enumerate}

\makeatletter
\g@addto@macro\bfseries{\boldmath} 
\makeatother

\usepackage{tikz-cd,graphicx}
\usepackage{tikz}
\usepackage{subcaption}
\usetikzlibrary{shapes.geometric, arrows}
\usetikzlibrary{positioning,arrows,calc}
\tikzstyle{startstop} = [rectangle, rounded corners, minimum width=2cm, minimum height=1cm,text centered,text width=2cm, draw=black]
\tikzstyle{io} = [rectangle, rounded corners, minimum width=2cm, minimum height=1cm,text centered,text width=2cm, draw=black]

\tikzstyle{arrow} = [thick,->,>=stealth]
\tikzstyle{arrow2} = [dashed,->,>=stealth]

\usepackage{dot2texi}

\numberwithin{equation}{section}
\usepackage{mathtools}

\usepackage[shortlabels]{enumitem}

\newcommand{\ep}{\varepsilon}

\newcommand{\N}{\mathbb{N}}

\newcommand{\dens}{\operatorname{dens}}

\newcommand{\eps}{\varepsilon}

\newcommand{\nnorm}[1]{{\left\vert\kern-0.25ex\left\vert\kern-0.25ex\left\vert #1%
    \right\vert\kern-0.25ex\right\vert\kern-0.25ex\right\vert}}

\newcommand{\Id}{{\mathrm{Id}}}

\newcommand{\Lip}{{\mathrm{Lip}}_0}

\DeclareMathOperator{\supp}{supp}
\DeclareMathOperator{\conv}{conv}

\DeclareMathOperator{\cconv}{\overline{conv}}

\DeclareMathOperator{\spn}{span}
\DeclareMathOperator{\cspn}{\overline{span}}

\newcommand{\R}{\mathbb{R}}
\renewcommand{\geq}{\geqslant}
\renewcommand{\leq}{\leqslant}

\theoremstyle{plain}
 \newtheorem{theorem}{Theorem}[section]
 \newtheorem{lemma}[theorem]{Lemma}
 \newtheorem{proposition}[theorem]{Proposition}
 \newtheorem{corollary}[theorem]{Corollary}

\theoremstyle{definition}
 \newtheorem{definition}[theorem]{Definition}
 \newtheorem*{definition*}{Definition}
 \newtheorem{remark}[theorem]{Remark}
 
 \newtheorem{question}[theorem]{Question}

\newcommand{\iten}{\,\ensuremath{\widehat{\otimes}_\varepsilon}\,}
\newcommand{\pten}{\,\ensuremath{\widehat{\otimes}_\pi}\,}

\begin{document}

\title{Transfinite Daugavet property}

\author[Avil\'es]{Antonio Avil\'es}
\address[Avil\'es]{Universidad de Murcia, Departamento de Matem\'{a}ticas, Campus de Espinardo 30100 Murcia, Spain
	\newline
	\href{https://orcid.org/0000-0003-0291-3113}{ORCID: \texttt{0000-0003-0291-3113} } }
\email{\texttt{avileslo@um.es}}

\author[Langemets]{Johann Langemets}
\address[Langemets]{Institute of Mathematics and Statistics, University of Tartu, Narva mnt 18, 51009 Tartu, Estonia}
\email{johann.langemets@ut.ee}
\urladdr{\url{https://www.johannlangemets.com/}}
\urladdr{
\href{https://orcid.org/0000-0001-9649-7282}{ORCID: \texttt{0000-0001-9649-7282} } }

\author[Mart\'in]{Miguel Mart\'in}
\address[Mart\'in]{Department of Mathematical Analysis and Institute of Mathematics (IMAG), University of Granada, E-18071 Granada, Spain}
\email{mmartins@ugr.es}
\urladdr{\url{https://www.ugr.es/local/mmartins/}}
\urladdr{
\href{http://orcid.org/0000-0003-4502-798X}{ORCID: \texttt{0000-0003-4502-798X} } }

\author[Rueda Zoca]{Abraham Rueda Zoca}
\address[Rueda Zoca]{Department of Mathematical Analysis and Institute of Mathematics (IMAG), University of Granada, E-18071 Granada, Spain}
\email{abrahamrueda@ugr.es}
\urladdr{\url{https://arzenglish.wordpress.com}}
\urladdr{
\href{https://orcid.org/0000-0003-0718-1353}{ORCID: \texttt{0000-0003-0718-1353} } }

\date{\today}


\begin{abstract}
We extend the Daugavet property  and a perfect version of it to transfinite cardinals in order to distinguish between spaces with the ordinary Daugavet property by some kind of complexity (topological, density\ldots), providing a number of examples and results. First, we characterise the transfinite Daugavet $C(K)$ spaces in terms of a cardinal index $\mathfrak r(K)$, which generalises the notion of the reaping number of a Boolean algebra. Besides, the perfect Daugavet property characterizes the absence of $G_\delta$-points in $K$. We also study several inheritance results of the transfinite Daugavet properties by almost isometric ideals, absolute sums, and tensor product spaces, with a number of applications. We classify these properties for $L_1(\mu)$ and $L_\infty(\mu)$ spaces in terms of the Maharam's decomposition of the measure. We also show that the space of Lipschitz functions $\Lip(M)$ on a complete length metric space has the $\omega$-perfect Daugavet property, improving the previous knowledge.
\end{abstract}

\maketitle
\begin{center}
\begin{minipage}{.8\textwidth}
  \centering
\parskip=0ex
    \tableofcontents
\end{minipage}
\end{center}

\section{Introduction}

The study of geometric properties of Banach spaces in the extreme opposite to the Radon-Nikod\'{y}m property or to the Fr\'echet smoothness is a very active research area since the appearance of the Daugavet property and octahedrality of norms in the last part of the XX Century, which received an extra boost with the origin of the big slice phenomena in the early 2000's. As a sample of the state of the art in this area, we refer the reader to \cite{almt21,amcps25,blr15eje2,cmp24,hllnr20,Samir, KaniaLewicki2026,lrz21} and references therein. Being the main object of study of this paper, we recall that a Banach space $X$ is said to satisfy the \emph{Daugavet property} if every rank-one bounded linear operator $T\colon X\longrightarrow X$ satisfies $\|\Id + T\|=1+\|T\|$ (equality known as the \textit{Daugavet equation}), where $\Id$ stands for the identity operator. We refer the reader to the very recent book \cite{kmrzw25} for reference and background. As a sample of ideas, let us just mention that the Daugavet property characterises the perfectness of $K$ on $C(K)$ spaces, the atomlessness of $\mu$ on $L_1(\mu)$ or $L_\infty(\mu)$ spaces, and the diffuseness of C$^*$-algebras and von Neumann preduals.

A very recent and emerging topic in the big slice phenomena is the study of generalisations to transfinite versions of Banach spaces properties as octahedrality, (symmetric) strong diameter two properties or almost squareness in the successive papers \cite{acllrz23,c23,cll23,cll22} (see also the PhD dissertation \cite{ciaciphd}). Let us point out, as a matter of fact, that this research line permitted to obtain multiple characterisations of the fact that a Banach space $X$ contains either $\ell_1(\kappa)$ or $c_0(\kappa)$ in the papers \cite{acllrz23, amcrz23,cll23}.

The aim of the present manuscript is to study transfinite versions of the Daugavet property. Let us motivate and present these versions, see Section~\ref{sect:notation} for notation and terminology. One of the equivalent formulations of the Daugavet property for a Banach space $X$ is that for every element $y$ in the unit sphere $S_X$ of $X$, every slice $S$ of the closed unit ball $B_X$ of $X$ and every $\varepsilon>0$, there exists $x\in S$ with $\|y+x\|>2-\varepsilon$ \cite[Theorem 3.1.5]{kmrzw25}. Based on this, we introduce the following definition. For a cardinal $\kappa$, we write $\{y_\alpha\colon \alpha<\kappa\}$ as a way to denote a set $A$ with $|A|\leq \kappa$. As usual, $\omega$ denotes the first infinite cardinal.

\begin{definition}\label{def:daugaprop}
Let $X$ be a Banach space and let $\kappa$ be a cardinal. We say that $X$ has the \textit{$\kappa$-Daugavet property} if for every family $\{y_\alpha\colon \alpha<\kappa\}\subseteq S_X$, every slice $S$ of $B_X$, and every $\varepsilon>0$, we can find $x\in S$ such that $\Vert y_\alpha+x\Vert>2-\varepsilon$
holds for all $\alpha<\kappa$.
\end{definition}

We are interested in studying the property for infinite $\kappa$, but the definition makes sense in any case. Every Banach space trivially has the $0$-Daugavet property, while the $1$-Daugavet property is just the equivalent reformulation of the Daugavet property given above. As a matter of fact, the ($1$-)Daugavet property actually implies the $n$-Daugavet property for every $n\in \N$ (see the proof of \cite[Lemma~3.1.14]{kmrzw25}, for instance).

We also introduce the transfinite version of the \textit{perfect Daugavet property}, a variant of the Daugavet property considered in \cite[Section 7.10]{kmrzw25} in which $\varepsilon$ is allowed to take the value zero and it is related to the possible stability of the Daugavet property by ultraproducts. 

\begin{definition}
Let $X$ be a Banach space and let $\kappa$ be a cardinal. We say that $X$ has the $\kappa$\textit{-perfect Daugavet property} if for every family $\{y_\alpha\colon \alpha<\kappa\}\subseteq S_X$ and every slice $S$ of $B_X$, we can find $x\in S$ such that
$\Vert y_\alpha+x\Vert=2$ holds for all $\alpha<\kappa$.
\end{definition}

Again, every Banach space has the $0$-perfect Daugavet property and the $1$-perfect Daugavet property is just the definition of the perfect Daugavet property given in \cite[Section 7.10]{kmrzw25}. Nevertheless, we do not know if the ($1$-)perfect Daugavet property implies the $2$-perfect Daugavet property.

It is common in topology to encode these ideas in the form of cardinal invariants. We define $\mathfrak{Dau}(X)$ as the least cardinal $\kappa$ such that $X$ fails the $\kappa$-Daugavet property, and  $\mathfrak{pDau}(X)$ as the least cardinal $\kappa$ such that $X$ fails the $\kappa$-perfect Daugavet property. Thus, a Banach space $X$ has the $\kappa$-Daugavet property if and only if $\kappa< \mathfrak{Dau}(X)$, and $X$ has the  $\kappa$-perfect Daugavet property if and only if $\kappa< \mathfrak{pDau}(X)$. These indices also encode the classical properties in the sense that $X$ fails the Daugavet property if and only if $\mathfrak{Dau}(X)=1$ (if and only if $\mathfrak{Dau}(X)<+\infty$) and, similarly, $X$ fails the perfect Daugavet property if and only if $\mathfrak{pDau}(X)=1$. Observe that these invariant cardinals measure the lowest size of a family that does not admit any almost $L$-orthogonal (in the case of $\mathfrak{Dau}(X)$) or even $L$-orthogonal (in the case of $\mathfrak{pDau}(X)$) inside some slice. In particular, observe further that $\mathfrak{Dau}(X)$  has to be smaller than or equal to the density character of the space $X$. Hence, for instance, a separable Banach space with the Daugavet property has $\mathfrak{Dau}(X)=\omega$. For non-separable Banach spaces, as we will see, these indices provide a way to distinguish between Banach spaces with the Daugavet property showing, somehow, that there are Banach spaces with stronger forms of the  property than others.

After providing some basic characterisations in Section~\ref{sect:firstproperties}, we study spaces of continuous functions $C(K)$, where the intended goal to distinguish between different levels of the Daugavet property takes form clearly. In Section~\ref{section:C(K)} we prove that, for infinite cardinals $\kappa$, the $\kappa$-Daugavet property, the $\kappa$-perfect Daugavet property, $\kappa$-octahedrality, and $\kappa$-rigid octahedrality, are all equivalent for $C(K)$ spaces. However, this is not the case for finite cardinals. We give a topological description of the cardinal $\mathfrak{Dau}(C(K))$ in terms of a topological invariant of the compact space $K$, which we call  the \textit{reaping number} of $K$ and we denote by $\mathfrak{r}(K)$. The reason for this name is that in the case when $K$ is totally disconnected, $\mathfrak{r}(K)$ coincides with the reaping number of the Boolean algebra of clopens of $K$, a well studied cardinal invariant \cite{bs1,bs2,burke,dsw,laflamme}, sometimes known under other names like \emph{weak density} or \emph{refinement number}. Direct application of some results in the literature give the value of $\mathfrak{Dau}(X)$ for some classical $C(K)$ spaces in terms of well known cardinals, like $\mathfrak{Dau}(\ell_\infty/c_0)=\mathfrak{r}$, $\mathfrak{Dau}(L_\infty[0,1]) = \operatorname{cof}\mathcal{N}$, $\mathfrak{Dau}(L_\infty[0,1]^\kappa) = \max( \operatorname{cof}\mathcal{N},\operatorname{cof}[\kappa]^\omega)$. We also show that $\mathfrak{Dau}(C(K))\geq \kappa$ whenever $K$ is a product of $\kappa$ many compact spaces with more than one point. Some of the known facts for the reaping number of Boolean algebras generalise naturally for compact spaces, but there are also cases when new subtler phenomena seem to appear. In Section~\ref{section:reapingnu}, rather topological in nature and of independent interest, we analyse this situation and the implications for the behavior of the $\kappa$-Daugavet property in $C(K)$ spaces.

In Section~\ref{sect:directsums}, we show that for $\ell_1$-sums and finite $\ell_\infty$-sums of Banach spaces, we can compute the Daugavet and perfect Daugavet cardinal invariants by the formulae $$\mathfrak{Dau}\left(\Big(\bigoplus\nolimits_{i\in I} X_i\Big)_E\right) = \min_{i\in I}\{\mathfrak{Dau}(X_i)\},  \quad \mathfrak{pDau}\left(\Big (\bigoplus\nolimits_{i\in I} X_i\Big)_E\right) = \min_{i\in I}\{\mathfrak{pDau}(X_i)\}$$
(where $E$ equals either $\ell_1(I)$ for arbitrary $I$ or $\ell_\infty(J)$ for finite $J$). 
In the case of infinite $\ell_\infty$-sums, we are only able to prove the inequalities $(\leq)$, except for the case of $C(K)$-spaces for which the results of Section~\ref{section:reapingnu} allow us to show the converse inequality as well. The formula fails for infinite $c_0$-sum.  Combining these formulae with Maharam's theorem and the previous results, one can compute $\mathfrak{Dau}(L_1(\mu))$ for all measures (in Section~\ref{sect:L1}) and $\mathfrak{Dau}(L_\infty(\mu))$ for localisable measures.

In Section~\ref{sect:L1} we prove that $\mathfrak{Dau}(L_1(\{0,1\}^\kappa)) = \kappa$ for infinite cardinals, leading to a calculation of the index for all $L_1(\mu)$ spaces with the help of Maharam's decomposition. By contrast, no perfect Daugavet property holds in any $L_1(\mu)$ space.

Section~\ref{sect:idealsandUD} provides another example of spaces enjoying transfinite Daugavet properties. We show that the $\kappa$-Daugavet property is inherited by $\kappa$-almost isometric ideals, and that spaces of almost universal disposition for subspaces of density $\kappa$ enjoy the $\kappa$-Daugavet property. These spaces, recently studied in \cite{mrz25}, are generalisations of the Gurari\u{\i} space and its transfinite cousins \cite{gk11}. Similar results hold for $\kappa$-perfect Daugavet property, $\kappa$-isometric ideals, and spaces of universal disposition.

We show in Section~\ref{sect:tensor} that the projective tensor product
$X\pten Y$ preserves the $\kappa$-Daugavet and the $\kappa$-perfect Daugavet properties when $X$ and $Y$ are also injective Banach spaces (or even satisfy a weaker injectivity condition). Concerning the injective tensor product, we have an intriguing result from Section~\ref{section:reapingnu} that if $K$ and $L$ are totally disconnected compact spaces, then 
 $$\mathfrak{Dau}(C(K))\vee \mathfrak{Dau}(C(L))\leq \mathfrak{Dau}\bigl(C(K)\iten C(L)\bigr) \leq \mathfrak{Dau}(C(K))^+\vee \mathfrak{Dau}(C(L))^+,$$
where we are denoting $\kappa\vee \kappa' := \max(\kappa,\kappa')$ and $\kappa^+$ is the successor cardinal of $\kappa$. The second inequality says that, for spaces of continuous functions, if $X\iten Y$ has the $\kappa$-Daugavet property, then either $X$ or $Y$ has the $\kappa^+$-Daugavet property. We do not know if this successor operation is necessary or if this holds for more general spaces $X$, $Y$.
 
Finally, in Section~\ref{sect:Lipschitz} we study the Banach space  $\Lip(M)$ of Lipschitz functions on a complete metric space $M$. The complete metric spaces for which $\Lip(M)$ has the Daugavet property are exactly the \textit{length metric spaces}, those where every couple points has an approximate midpoint, cf.\ \cite{bbi}. We prove that if $M$ is a length metric space then, in fact, $\Lip(M)$ has the $\omega$-perfect Daugavet property. This result improves both \cite[Theorem 5.2]{cll23} and \cite[Theorem 3.5]{gpr18}.

Finally, we conclude the paper with Section~\ref{sect:remarkopeque}, which is devoted to present a collection of remarks and open questions.

\section{Notation and preliminary results}\label{sect:notation}
Even though most of the results do not depend on the base field, we will only consider in this paper real Banach spaces. Our notation is completely standard and follows textbooks as \cite{FHHMZ2011}. Given a (real) Banach space $X$, $B_X$ (respectively, $S_X$) stands for the closed unit ball (respectively, the unit sphere) of $X$. We will denote by $X^*$ the topological dual of $X$. Given a subset $C$ of $X$, $\conv(C)$ is the convex hull of $C$. We will say that a Banach space $X$ has \emph{density} $\kappa$, denoted by $\dens(X)$, if it is the least cardinal that a subset spanning a dense subspace of $X$ can have. With this definition we recover the classical one for infinite dimensional Banach spaces, that is, the least cardinal of a dense subset of $X$, but have the advantage that for a finite dimensional Banach space $X$ we obtain that $\dens(X) = \dim(X)$.

If $C$ is a bounded (and not necessarily convex) subset of a Banach space $X$, by a \textit{slice} of $C$ we will mean the non-empty intersection of an open half-space with the bounded  set $C$. Every slice can be written in the form 
$$S(C,f,\alpha):=\{x\in C\colon  f(x)>\sup f(C)-\alpha\}$$
for suitable $f\in X^*$ and $\alpha>0$. Notice that when $C$ is convex, every slice $S$ of $C$ is convex and so is $C\setminus S$.

Given a family $\{X_i\colon i\in I\}$ of Banach spaces, we write
$$
\left(\bigoplus_{i\in I}X_i\right)_{\ell_1},\qquad \left(\bigoplus_{i\in I}X_i\right)_{\ell_\infty}
$$
to denote, respectively, the $\ell_1$-sum and the $\ell_\infty$-sum of the family.

\subsection{Set theory and topology} We work in ZFC. Following the standard von Neumann's approach, each ordinal equals the set of preceding ordinals, $\alpha = \{\beta \colon \beta<\alpha\}$, and cardinals are ordinals that are not in bijection with any smaller ordinal. In this way, for example, $\{y_\alpha \colon \alpha<\kappa\}$ is just a family indexed by a set of cardinality $\kappa$ and the first infinite ordinal $\omega$ is also the set of natural numbers. All topological spaces are assumed to be Hausdorff. A set in a topological space is \textit{clopen} if it is closed and open simultaneously. A compact space is \textit{totally disconnected} (or \textit{zero-dimensional}) if every two different points are separated by a clopen set. A \textit{Boolean algebra} is a set endowed with abstract operations of union, intersection and complementation that make it isomorphic to an algebra of subsets of a set $\Omega$ closed under these set operations. The algebra of clopens of a topological space is an example, and Stone duality states that every Boolean algebra is isomorphic to the algebra of clopens of a uniquely determined totally disconnected compact space, its \textit{Stone space}.  We refer to \cite{jech} for set theory, \cite{engelking} for general topology, and \cite{walker} for Stone duality in particular.

\subsection{The Daugavet property and related geometric properties}
As we already said in the introduction, a Banach space $X$ is said to have the \textit{Daugavet property} if given any rank-one (bounded linear) operator $T\colon X\longrightarrow X$ the equality
$$\Vert\Id + T\Vert=1+\Vert T\Vert$$
holds. In \cite{kssw} the following useful characterisation in geometric terms was proved: a Banach space $X$ has the Daugavet property if and only if, given any finite subset $F\subseteq S_X$, any $\varepsilon>0$, and any slice $S$ of $B_X$, we can find $y\in S$ such that $\Vert x+y\Vert>2-\varepsilon$ holds for every $x\in F$. Note that this is equivalent to saying that for every finite-dimensional subspace $F$ of $X$, every slice $S$ of $B_X$, and every $\varepsilon>0$, there is $y\in S$ such that $$\|z+ty\|\geq (1-\varepsilon)\bigl(\|z\|+|t|\bigr)$$
for every $t\in \R$ and every $z\in F$; it is easily obtained from here that $X$ contains almost-isometric copies of $\ell_1$.
We refer the reader to \cite[Chapter 3]{kmrzw25} for background and more reformulations of the Daugavet property.

Closely related to the above reformulation of the Daugavet property is the property of \textit{octahedrality}. A Banach space $X$ is said to be \textit{octahedral} (or to have \textit{an octahedral norm)} if given any finite subset $F\subseteq S_X$ and any $\varepsilon>0$ we can find $y\in S_X$ such that $\Vert x+y\Vert>2-\varepsilon$ holds for every $x\in F$. This property, which was introduced at the end of the 1980s by G.\ Godefroy and B.\ Maurey, has recently received intense attention after the exhibited connection between octahedral norms and the diameter two properties proved in \cite{blr14}. For a brief survey about the octahedral norms and properties around we refer the reader to \cite[Section 12.2]{kmrzw25} and to references therein.

Natural generalisations of octahedrality in the way of defining transfinite properties were introduced in \cite{cll23}. Given a Banach space $X$ and a cardinal $\kappa$, we say that $X$ is:
\begin{enumerate}
    \item \textit{$\kappa$-octahedral} if for any subset $\{y_\alpha\colon \alpha<\kappa\}$ of $S_X$ and any $\varepsilon>0$, there exists $x\in S_X$ such that
    $$\Vert y_\alpha\pm x\Vert>2-\varepsilon$$
    holds for every $\alpha<\kappa$.
    \item \textit{$\kappa$-rigid octahedral} if for any subset $\{y_\alpha\colon \alpha<\kappa\}$ of $S_X$ there exists $x\in S_X$ such that
    $$\Vert y_\alpha\pm x\Vert=2$$
    holds for every $\alpha<\kappa$.
\end{enumerate}
For any infinite cardinal $\kappa$, writing $\|y_\alpha\pm x\|$ or $\|y_\alpha + x\|$ is irrelevant because we can suppose that the subset of $S_X$ is stable under switching signs. But for finite $\kappa$, the natural property is with $\pm$. In a similar way as we did for the Daugavet property, we define a cardinal invariant $\mathfrak{oct}(X)$ as the least cardinal $\kappa$ for which a Banach space $X$ fails to be $\kappa$-octahedral, and $\mathfrak{rig}\text{-}\mathfrak{oct}(X)$ is analogously defined. In this case, $X$ is octahedral if and only if $\mathfrak{oct}(X)\geq \omega$. However, $\mathfrak{oct}(X)>1$ is equivalent to the property called \emph{local octahedrality} \cite{hlp15}, cf.\ \cite[Section 12.2]{kmrzw25}. The notation of rigid octahedrality, which is introduced under the name of ``failure of the $(-1)$BCP$_{\kappa}$'' in \cite{cll23}, is taken from \cite{ciaciphd}. We refer the reader to \cite{amcrz23,cll23} for background and results about transfinite octahedrality properties. For more transfinite extensions of properties related to octahedrality and diameter two properties, we refer the reader to the aforementioned \cite{ciaciphd}. 

\section{Characterisations of the transfinite Daugavet properties}\label{sect:firstproperties}
This section is devoted to obtain multiple reformulations of the $\kappa$-Daugavet and $\kappa$-perfect Daugavet properties which will be of interest in the following sections.

Let us start with the following result which allows us to describe the transfinite Daugavet properties in terms of finding ``$L$-orthogonal vectors'' and ``almost $L$-orthogonal vectors'' to subspaces of certain prescribed density.

\begin{proposition}\label{lemma: subspace Daugavet}
Let $X$ be a Banach space and let $\kappa$ be an infinite cardinal. Then:
\begin{enumerate}
\item $X$ has the $\kappa$-Daugavet property if and only if for every subspace $Y$ of $X$ of density character at most $\kappa$, every slice $S$ of $B_X$ and every $\varepsilon>0$ we can find $x\in S$ such that the inequality
\[
\Vert y+rx\Vert\geq (1-\varepsilon)(\Vert y\Vert+\vert r\vert)
\]
holds for every $y\in Y$ and $r\in\R$.
\item $X$ has the $\kappa$-perfect Daugavet property if and only if for every subspace $Y$ of $X$ of density character at most $\kappa$ and every slice $S$ of $B_X$ we can find $x\in S$ such that the equality
\[
\Vert y+rx\Vert=\Vert y\Vert+\vert r\vert
\]
holds for every $y\in Y$ and $r\in\R$.
\end{enumerate}
\end{proposition}

\begin{proof} Let us prove (1), being the proof of (2) completely analogous. Moreover, in the proof of (1), one direction is obvious. So assume $X$ has the $\kappa$-Daugavet property and fix a subspace $Y$ of $X$ with $\dens(Y)\leq \kappa$, a slice $S$ of $B_X$, and $\varepsilon>0$. Let $D$ be a dense set in $Y$ of cardinal $\kappa$ and set
    \[
    D':=\left\{\frac{y}{\Vert y\Vert}\colon y\in (D\cup (-D))\setminus\{0\} \right\}.
    \]
    By our assumption, there is an $x\in S$ such that 
    \[
   \left \Vert  \frac{y}{\Vert y \Vert}+x \right\Vert > 2-\varepsilon   
    \]
    holds for all $\frac{y}{\Vert y \Vert}\in D'$.
    Hence, for every $\frac{y}{\Vert y \Vert}\in D'$ we can find an $x^*_y\in S_{X^*}$ such that 
    \[
    x^*_y(y)\geq (1-\varepsilon)\Vert y\Vert\qquad \text{and} \qquad x^*_y(x)\geq 1-\varepsilon.
    \]
    Now, let $y\in D$ and $r\in \R$. Suppose first that $r\geq 0$. Then
    \[
    \Vert y+rx\Vert \geq x^*_y(y)+rx^*_y(x)\geq (1-\varepsilon)(\Vert y\Vert +r)=(1-\varepsilon)(\Vert y\Vert +|r|).
    \]
    Suppose now that $r<0$. Then
     \[
    \Vert y+rx\Vert = \Vert -y-rx\Vert \geq x^*_{-y}(-y)-rx^*_{-y}(x)\geq (1-\varepsilon)(\Vert y\Vert -r)=(1-\varepsilon)(\Vert y\Vert +|r|).\qedhere
    \]
    This proves that the thesis os (1) holds true for every $y\in D$. Since $D$ is norm-dense in $Y$, we conclude the statement.
\end{proof}

Our next aim is to show that we can replace slices with non-empty weakly open subsets or with convex combinations of slices. Indeed, following ideas behind \cite[Lemma 2]{shv}, where such result was obtained for the classical Daugavet property, we obtain the following result.

\begin{lemma}[Transfinite Shvidkoy's lemma]\label{lemma: Daugavet and CCS}
    Let $X$ be a Banach space and let $\kappa$ be an infinite cardinal. Then the following are equivalent:
    \begin{itemize}
        \item[(i)] $X$ has the $\kappa$-Daugavet property;
        \item[(ii)] for every family $\{y_\alpha\colon\alpha<\kappa\}\subseteq S_X$, every nonempty relatively weakly open set $U$ of $B_X$ and every $\varepsilon>0$ we can find $x\in U$ such that
$$\Vert y_\alpha+x\Vert>2-\varepsilon$$
holds for every $\alpha<\kappa$.
        \item[(iii)] for every family $\{y_\alpha\colon\alpha<\kappa\}\subseteq S_X$, every finite convex combination of slices $C$ of $B_X$ and every $\varepsilon>0$ we can find $x\in C$ such that
$$\Vert y_\alpha+x\Vert>2-\varepsilon$$
holds for every $\alpha<\kappa$.
    \end{itemize}
\end{lemma}
\begin{proof}
    (i)$\implies$(iii). Let $\{y_\alpha\colon\alpha<\kappa\}\subseteq S_X$, let $C=\sum_{i=1}^n\lambda_iS_i$, where $S_i$ are slices of $B_X$, and $\varepsilon>0$. Pick $\delta\in(0,1)$ such that $2(1-\delta)^n>2-\varepsilon$. 

    Set $Y_0:=\cspn\{y_\alpha\colon\alpha<\kappa\}$. By Proposition~\ref{lemma: subspace Daugavet}, we can find $x_1\in S_1\cap S_X$ such that
    \[
    \|y+\lambda_1x_1\|\geq (1-\delta)(\|y\|+\lambda_1)
    \]
    holds for every $y\in Y_0$.

    Set now $Y_1:=\cspn(Y_0\cup\{x_1\})$. Again, by Proposition~\ref{lemma: subspace Daugavet}, we can find 
    $x_2\in S_2\cap S_X$ such that
    \[
    \|y+\lambda_2x_2\|\geq (1-\delta)(\|y\|+\lambda_2)
    \]
    holds for every $y\in Y_1$.

    By continuing this way, we can find $x_i\in S_i\cap S_X$ for every $i\in\{1,\dots,n\}$. Hence, $\sum_{i=1}^n\lambda_ix_i\in \sum_{i=1}^n\lambda_iS_i$ and, if we fix $\alpha$, then
    \begin{align*}
   \Big\Vert y_\alpha+\sum_{i=1}^n\lambda_ix_i\Big\Vert&=\Big\Vert \Big(y_\alpha+\sum_{i=1}^{n-1}\lambda_ix_i\Big)+\lambda_nx_n\Big\Vert\geq (1-\delta)\Big(\Big  \Vert y_\alpha+\sum_{i=1}^{n-1}\lambda_ix_i\Big\Vert+\lambda_n\Big)\\
   &\geq\ldots\geq(1-\delta)^n\Big(1+\sum_{i=1}^n\lambda_i\Big)=2(1-\delta)^n>2-\varepsilon.
    \end{align*}

(iii)$\implies$(ii). Follows from Bourgain's lemma, which says that every nonempty relatively weakly open subset of $B_X$ contains some finite convex combination of slices.

(ii)$\implies$(i). Is clear, because every slice is a relatively weakly open subset.
\end{proof}

With an adaptation of the above proof we can obtain a version for the $\kappa$-perfect Daugavet property in the following terms

\begin{proposition}\label{lemma: perfect Daugavet and CCS}
    Let $X$ be a Banach space and let $\kappa$ be an infinite cardinal. Then the following are equivalent:
    \begin{itemize}
        \item[(i)] $X$ has the $\kappa$-perfect Daugavet property;
        \item[(ii)] for every family $\{y_\alpha\colon\alpha<\kappa\}\subseteq S_X$ and every nonempty relatively weakly open set $U$ of $B_X$  we can find $x\in U$ such that
$$\Vert y_\alpha+x\Vert=2$$
holds for every $\alpha<\kappa$.
        \item[(iii)] for every family $\{y_\alpha\colon\alpha<\kappa\}\subseteq S_X$ and every finite convex combination of slices $C$ of $B_X$ we can find $x\in C$ such that
$$\Vert y_\alpha+x\Vert=2$$
holds for every $\alpha<\kappa$.
    \end{itemize}
\end{proposition}

Let us rewrite the above two characterisations using nets.

\begin{corollary}\label{cor:weaknetdauga}
Let $X$ be a Banach space. Then 
\begin{enumerate}
\item $X$ has the $\kappa$-Daugavet property if, and only if, for every family $\{y_\alpha\colon\alpha<\kappa\}\subseteq S_X$ and every $x_0\in B_X$ there exists a net $\{x_i\}_{i\in I}$ of points of $B_X$ weakly converging to $x_0$ such that $\limsup\limits_{i\in I} \,\inf\limits_{\alpha<\kappa} \Vert y_\alpha+x_i\Vert=2$ holds for every $\alpha<\kappa$.
\item $X$ has the $\kappa$-perfect Daugavet property if, and only if, for every family $\{y_\alpha\colon\alpha<\kappa\}\subseteq S_X$ and every $x_0\in B_X$ there exists a net $\{x_i\}_{i\in I}$ of points of $B_X$ weakly converging to $x_0$ such that $\Vert y_\alpha+x_i\Vert=2$ holds for every $\alpha<\kappa$ and every $i\in I$
\end{enumerate}
\end{corollary}

We conclude the section by exhibiting characterisations of the Daugavet and the perfect Daugavet properties which avoid the use of slices and will be used in the proof of Proposition~\ref{prop:hereaiidealargeDauga}. The proof follows similar lines to the analogous characterisation for the regular Daugavet property (c.f.\ e.g.\ \cite[Lemma 3.1.9]{kmrzw25}).

\begin{proposition}\label{prop:charlargeDPconvexhull}
Let $X$ be a Banach space. Then:
\begin{enumerate}
    \item $X$ has the $\kappa$-Daugavet property if, and only if, for every family $\{y_\alpha\colon\alpha<\kappa\}\subseteq S_X$ and every $\varepsilon>0$ the following condition holds
    $$B_X=\cconv\bigl(\{x\in B_X\colon  \forall \alpha<\kappa\ \Vert y_\alpha-x\Vert>2-\varepsilon\}\bigr).$$
     \item $X$ has the $\kappa$-perfect Daugavet property if, and only if, for every family $\{y_\alpha\colon\alpha<\kappa\}\subseteq S_X$ the following condition holds
    $$B_X=\cconv\bigl(\{x\in B_X\colon  \forall \alpha<\kappa \ \Vert y_\alpha-x\Vert=2\}\bigr).$$
\end{enumerate}
\end{proposition}

\section{\texorpdfstring{$C(K)$}{C(K)} spaces}\label{section:C(K)}
In this section we will characterise the $\kappa$-Daugavet and $\kappa$-perfect Daugavet properties of a space of continuous functions $C(K)$ in terms of the compact space $K$. We start by the more classical case of finite $\kappa$. Recall that, for every Banach space $X$, if $\mathfrak{Dau}(X)<\omega$, then $\mathfrak{Dau}(X)=1$, see \cite[Lemma 3.1.14]{kmrzw25}. We do not know if the same holds for the index $\mathfrak{pDau}(X)$, but it is at least the case for $C(K)$ spaces as shown in Theorem~\ref{theo:CKDaugavet} below.

Let us present a bit of topological preliminaries. Recall that a set in a topological space is $G_\delta$ if it is a countable intersection of open sets. A point $x$ is called a \emph{$G_\delta$-point} if $\{x\}$ is a $G_\delta$ set. In our setting, when $K$ is a compact topological space, $x$ is a $G_\delta$-point if and only if there exists $f\colon K\longrightarrow [0,1]$ such that $\{x\}=f^{-1}(\{0\})$, \cite[Corollary 1.5.12]{engelking}. By considering $g=1-f$ or $g=1-2f$, we can also say that $x$ is a $G_\delta$-point if and only if $\{x\}=g^{-1}(\{1\})$ for some $g\colon K\longrightarrow [0,1]$ or, equivalently, for some $g\colon K\longrightarrow [-1,1]$. Another observation that we shall use in the proof below is that if we have a continuous function $f\colon K\longrightarrow \mathbb{R}$ and an open set $W$ in $K$, and $F = W\cap f^{-1}(\{t\})$ was finite, then all points in $F$ would be $G_\delta$-points. Indeed, for $x\in F$, take an open neighborhood $V\subset W$ of $x$ that separates it from the rest of points of $F$ and notice that $\{x\} = V\cap \bigcap_{n<\omega}f^{-1}\bigl((t-1/n,t+1/n)\bigr)$.
Observe that it is immediate that in a metric space, all points are $G_\delta$.

\begin{theorem}\label{theo:CKDaugavet}
    Let $K$ be a compact space.
    \begin{enumerate}
        \item $C(K)$ has the Daugavet property if and only if $K$ has no isolated points. Moreover, if $K$ has an isolated point, then $\mathfrak{oct}(C(K))=1$.
        \item $C(K)$ has the perfect Daugavet property if and only if $K$ has no $G_\delta$-points. Moreover, in this case, $\mathfrak{pDau}(C(K))\geq \omega$. 
    \end{enumerate}
\end{theorem}

\begin{proof}
Part (1) is known, see \cite[Theorem 3.3.1]{kmrzw25}. Indeed, if $K$ has an isolated point $a$, then the characteristic function $y=\chi_{\{a\}}$ of that point satisfies that for each $f\in B_{C(K)}$ either $\|y+ f\|\leq 1$ (if $f(a)\geq 0$) or $\|y-f\|\leq 1$ (if $f(a)\leq 0$), and this shows that $\mathfrak{oct}(C(K))=1$. For part (2) suppose first that $C(K)$ has a $G_\delta$-point $x$ and we prove that $K$ fails the perfect Daugavet property, and hence $\mathfrak{pDau}(C(K))=1<\omega$. Then there exists a continuous function $f\colon K\longrightarrow [0,1]$ such that $f^{-1}(\{1\})=\{x\}$. Consider the slice $S=\{g\in B_{C(K)} \colon g(x)<-1/2\}$. We check that there is no $g\in S$ such that $\|g+f\|=2$. Fix $g\in S$. First, observe that $f(y)+g(y)\geq 0-1 = -1$ for all $y\in K$. Second, notice that $f(x)+g(x)<1-1/2=1/2$ since $g\in S$. And third, $f(y)+g(y)<1+1=2$ if $y\neq x$ because $f^{-1}(\{1\})=\{x\}$. So we conclude that $\|f+g\|<2$. Now, for the converse, we suppose that $K$ has no $G_\delta$-points and we prove that $\mathfrak{pDau}(C(K))\geq\omega$ and, in particular, that $C(K)$ has the perfect Daugavet property. We have to show that $C(K)$ has the $n$-perfect Daugavet property for every $n$, so fix $f_1,\ldots,f_n\in S_{C(K)}$ and $S=\{g\in B_{C(K)} \colon \int_K g \cdot d\mu>1-\varepsilon\}$ a slice of the ball, given by a finite regular Borel measure $\mu$ by virtue of the Riesz representation theorem for $C(K)^*$. Normalising, we assume that $\|\mu\|=1$ and $\varepsilon>0$. For $i=1,\ldots,n$ we have a point $x_i\in K$ and $\eta_i\in\{-1,1\}$ where $f_i(x_i)=\eta_i$. Since $K$ has no $G_\delta$ points, the sets $f_i^{-1}(\{\eta_i\})$ are infinite, so we can suppose that the points $x_1,\ldots,x_n$ are all different. By the Hahn decomposition theorem of measures applied to $\mu|_{K\setminus\{x_1,\ldots,x_n\}}$, we can write $K\setminus\{x_1,\ldots,x_n\} = A\cup B$ as the disjoint union of two Borel sets $A,B$ such that $\mu|_A$ is a positive measure and $\mu|_B$ is a negative measure. Since $\mu$ is a regular measure, we can find two closed sets $A_0\subseteq A$ and $B_0\subseteq B$ such that $|\mu|(A\setminus A_0)<\varepsilon/4$ and  $|\mu|(B\setminus B_0)<\varepsilon/4$. Again, since $K$ has no $G_\delta$-points, the sets $f_i^{-1}(\{\eta_i\}) \cap [K\setminus(A_0\cup B_0)]$ must be infinite (they are nonempty since $x_i$ is there). So we can pick points $y_i\in K\setminus(A_0\cup B_0)$ which are all different (and different from the $x_i$'s) such that $f_i(y_i)=\eta_i$. By Tietze's extension theorem, there exists a continuous function $g\colon K\longrightarrow [-1,1]$ such that 
$$g|_{A_0\cup\{y_i \colon \eta_i=1\}}=1, \ g|_{B_0\cup\{y_i \colon \eta_i=-1\}}=-1 \text{ and } g(x_i) = \begin{cases}
    +1 & \text{ if }\mu(\{x_i\})\geq 0,\\
    -1 & \text{ if } \mu(\{x_i\})<0.
\end{cases}$$ 
First, $\|f_i+g\|=2$ because $f_i(y_i)=g(y_i)=\eta_i$. To show that $g\in S$, the computation is
\begin{align*}
\int_K g\cdot d\mu &= \int_{A_0} 1 \cdot d\mu + \int_{B_0}(-1) \cdot d\mu + \int_{\{x_1,\ldots,x_n\}} g \cdot d\mu + \int_{K\setminus (A_0\cup B_0\cup\{x_1,\ldots,x_n\})}g \cdot d\mu \\ &=|\mu|(A_0) + |\mu|(B_0)+ \sum_{i=1}^n|\mu|(\{x_i\}) + \int_{K\setminus (A_0\cup B_0\cup\{x_1,\ldots,x_n\})}g\cdot d\mu\\& > (\|\mu\| - \varepsilon/2) -\varepsilon/2 = 1-\varepsilon.\qedhere
\end{align*}
\end{proof}

\begin{remark}\label{remark:C(K)Gdeltarigidoh}
If $K$ has a $G_\delta$ point then $\mathfrak{rig}\text{-}\mathfrak{oct}(C(K))=1$. Indeed, under the assumption, there exists a continuous function $f\colon K\longrightarrow [0,1]$ such that $f^{-1}(\{1\})=\{x\}$ for some $x\in K$. Now, given $g\in B_{C(K)}$, we claim that $\min\{\Vert f+g\Vert,\Vert f-g\Vert\}<2$. Indeed, if $\Vert f+g\Vert=2$ then there exists $t\in K$ such that $\vert f(t)+g(t)\vert=2$, which implies $\vert f(t)\vert=1=\vert g(t)\vert$. By the condition on $f$ we infer $t=x$. Since $\vert f(x)+g(x)\vert=2$ then $g(x)=1$. Now, we get that $\vert f(x)-g(x)\vert=0$, whereas if $t\neq x$ we get $\vert f(t)-g(t)\vert\leq \vert f(t)\vert+\vert g(t)\vert\leq 1+f(t)<2$ since $t\neq x$. Now $\Vert f-g\Vert=\max\limits_{t\in K}\vert f(t)-g(t)\vert<2$, as desired.
\end{remark}

In order to study the transfinite Daugavet property for $C(K)$ spaces, the following is the key topological cardinal invariant.

\begin{definition}\label{def:reapingnumber}
    Given a compact space $K$, the cardinal $\mathfrak{r}(K)$, that we call the \emph{reaping number} of $K$, is the least cardinality of a family $\mathcal{F}$ of nonempty open subsets of $K$ such that for every two disjoint closed sets $L_1$ and $L_2$ there exists $W\in\mathcal{F}$  such that either $W\cap L_1 = \emptyset$ or $W\cap L_2=\emptyset$.
\end{definition}

We state a lemma that we need now, but we will see later that there is a more general fact behind it, see Remark~\ref{remark Laflamme new}.

\begin{lemma}\label{rnn}
    Given a  compact topological space $K$, the reaping number $\mathfrak{r}(K)$ equals the least cardinality of a family $\mathcal{F}$ of nonempty open sets such that for every three pairwise disjoint closed sets $L_1,L_2,L_3$ there exists $W\in\mathcal{F}$ such that $W\cap L_i = \emptyset$ for some $i\in\{1,2,3\}$.
\end{lemma}

\begin{proof}
    Let $\ddot{\mathfrak{r}}_{3,3}(K)$ be the cardinality defined in the statement of the Lemma. Such a notation will be justified later. If a family $\mathcal{F}$ is as in Definition~\ref{def:reapingnumber}, then $\mathcal{F}$ is also as in the lemma, just  by considering two disjoint closed sets $L_1$ and $L_2\cup L_3$. This proves that $\ddot{\mathfrak{r}}_{3,3}(K)\leq \mathfrak{r}(K)$. If this inequality was strict, then we would have a family $\mathcal{F}$ with $|\mathcal{F}|< \mathfrak{r}(K)$ that satisfies the statement of the lemma for three closed sets.  Since $|\mathcal{F}|<\mathfrak{r}(K)$, we have disjoint closed sets $L_1$ and $L_2$ such that every element of $\mathcal{F}$ intersects both $L_1$ and $L_2$. Consider an open set $W\supset L_1$ whose closure is disjoint from $L_2$. Since $\{A\cap W \colon A\in\mathcal{F}\}$ is a family of nonempty open sets of cardinality less than $\mathfrak{r}(K)$, there exist two disjoint closed sets $L'_0$ and $L'_1$ that intersect all $A\cap W$ for $A\in\mathcal{F}$. The sets $L'_0\cap \overline{W}, L'_1\cap \overline{W}$ and $L_2$
    are three pairwise disjoint closed sets that intersect all elements of $\mathcal{F}$, which contradicts that $\mathcal{F}$ is as in the lemma.
\end{proof}

 We are in a position to state the main result of this section.

\begin{theorem}\label{theo:CK}
   If $K$ is a compact space, then $$ \mathfrak{pDau}(C(K))=\mathfrak{Dau}(C(K)) = \mathfrak{oct}(C(K))=\mathfrak{rig}\text{-}\mathfrak{oct}(C(K))=\mathfrak{r}(
    K)$$ 
    unless $K$ is a compact space with a $G_\delta$-point but without any isolated point, in which case $$\mathfrak{pDau}(C(K))=\mathfrak{rig}\text{-}\mathfrak{oct}(C(K)) = 1 <\omega = \mathfrak{Dau}(C(K)) = \mathfrak{oct}(C(K))=\mathfrak{r}(
    K).$$
\end{theorem}

\begin{proof} We have that $\mathfrak{pDau}(X)\leq \mathfrak{Dau}(X)\leq \mathfrak{oct}(X)$ and $\mathfrak{pDau}(X)\leq \mathfrak{rig}\text{-}\mathfrak{oct}(X)\leq \mathfrak{oct}(X)$ hold for any Banach space $X$. So we show (under suitable assumptions that will be indicated) that  $$\mathfrak{oct}(C(K))\leq \mathfrak{r}(K) \leq \mathfrak{pDau}(C(K)).$$ We prove the first inequality $\mathfrak{oct}(C(K))\leq \mathfrak{r}(K)$ under the assumption that $\mathfrak{r}(K)\geq \omega$. For this, we take an infinite $\kappa<\mathfrak{oct}(C(K))$  and we prove that $\kappa<\mathfrak{r}(K)$. The latter inequality is the same as saying that for any collection of non-empty open sets $\mathcal{F}=\{O_\alpha\colon \alpha<\kappa\}$ there are two disjoint closed sets  $L_1$ and $L_2$ with $O_\alpha\cap L_i\neq \emptyset$ for every $\alpha<\kappa$ and $i\in\{1,2\}$. By Urysohn's Lemma, for every $\alpha<\kappa$ we can find a continuous function $f_\alpha\colon K\longrightarrow [0,1]$ that attains its maximum 1 inside $O_\alpha$ and vanishes on $K\setminus O_\alpha$. Since $\kappa<\mathfrak{oct}(C(K))$, we know that $C(K)$ is $\kappa$-octahedral, so we can find $g\in S_{C(K)}$ such that
$$\Vert f_\alpha\pm g\Vert>2-\frac{1}{8}\ \text{ for all }\alpha<\kappa.$$
We define
$$L_1:=\left\{t\in K\colon  g(t)\geq 1-\frac{1}{8}\right\}, \hspace{1cm} L_2:=\left\{t\in K\colon g(t)\leq -1+\frac{1}{8}\right\}.$$
These are two disjoint closed sets. It remains to check that $L_i\cap O_\alpha\neq \emptyset$ for every $\alpha<\kappa$ and $i\in \{1,2\}$. We consider the case $i=1$. So fix $\alpha<\kappa$. Since $\Vert f_\alpha+g\Vert>2-\frac{1}{8}$ we can find $t_\alpha\in K$ such that
$$\vert f_\alpha(t_\alpha)+g(t_\alpha)\vert>2-\frac{1}{8}.$$
But $\vert f_\alpha(t_\alpha)+g(t_\alpha)\vert= f_\alpha(t_\alpha)+g(t_\alpha)$ because
$-f_\alpha(t_\alpha)-g(t_\alpha)\leq 0-g(t_\alpha)\leq 1$
since $f_\alpha\geq 0$. So
$$f_\alpha(t_\alpha)+g(t_\alpha)>2-\frac{1}{8},$$ which implies that $$f_\alpha(t_\alpha)>1-\frac{1}{8}\mbox{ and }g(t_\alpha)>1-\frac{1}{8}$$
using that both $f_\alpha$ and $g$ are norm-one functions. Now, on the one hand, since $g(t_\alpha)>1-\frac{1}{8}$ implies $t_\alpha\in L_1$ from the very definition of $L_1$. On the other hand, the fact that $f_\alpha(t_\alpha)>1-\frac{1}{8}>0$ implies by the choice of $f_\alpha$ that $t_\alpha\in O_\alpha$. This proves $O_\alpha\cap L_1\neq \emptyset$, as desired. The proof that $L_2\cap O_\alpha\neq\emptyset$ is similar, using that $\Vert f_\alpha-g\Vert>2-\frac{1}{8}$.

Now we prove that $\mathfrak{r}(K)\leq \mathfrak{pDau}(C(K))$ under the assumption that $\mathfrak{pDau}(C(K))\geq\omega$. For this, it is enough to take an infinite $\kappa<\mathfrak{r}(K)$ and then prove that $\kappa<\mathfrak{pDau}(C(K))$, which means that $C(K)$ has the $\kappa$-perfect Daugavet property. Following Corollary~\ref{cor:weaknetdauga}, we fix a family of functions $\{f_\alpha\colon \alpha<\kappa\}\subseteq S_{C(K)}$ and some $g\in B_{C(K)}$, and we will find a sequence $\{h_n\}_{n\in\mathbb{N}}$ in $B_{C(K)}$ that weakly converges to $g$ such that $\Vert f_\alpha+h_n\Vert=2$ holds for every $\alpha<\kappa$ and every $n\in\mathbb N$.

For every $\alpha$ we pick $\varepsilon_\alpha\in\{-,+\}$ such that $\{t\in K\colon   \varepsilon_\alpha f_\alpha(t)=1\}$ is non-empty for every $\alpha$. Consequently, given $\alpha<\kappa$ and $k\in\mathbb N$, the open set
$$W_\alpha^k:=\left\{t\in K\colon \varepsilon_\alpha f_\alpha(t)>1-\frac{1}{k}\right\}$$
is non-empty for every $k\in\mathbb N$. 

We make an inductive procedure that will create two descending chains of nonempty open sets $K=V'_0 \supseteq V'_1 \supseteq V'_2 \supseteq \cdots$ and $K=V_0 \supseteq V_1 \supseteq V_2 \supseteq \cdots$ with $\overline{V}_n\subseteq V'_n$, in such a way that at step $n$ we are given $V'_{n-1}$ and $V_{n-1}$ such that $\{V_{n-1}\cap W_\alpha^k\}$ is a family of nonempty open subsets of $K$ of cardinality $\leq \kappa$.

By Lemma~\ref{rnn}, since $\kappa<\mathfrak{r}(K)$, we can find three pairwise disjoint closed sets $L^{-}_n$, $L^0_n$ and $L^+_n$ that have nonempty intersection with all $V_{n-1}\cap W_\alpha^k$, for $\alpha<\kappa$, $k\in\mathbb{N}$. We can suppose that $L_n^-\cup L^0_n\cup L_n^+ \subseteq \overline{V}_{n-1}$. We can find three continuous functions $\psi_n^{-},\psi_n^0,\psi_n^+\in C(K)$ with pairwise disjoint supports and such that $\psi^i_n|_{L_n}$ is constant equal to 1 for $i\in\{-,0,+\}$. Moreover, we can assume that $\psi^0_n$ is constant equal to 1 on an open set $V'_n$ that contains $L^0_n\cup K\setminus V'_{n-1}$ and is disjoint from $L_n^+\cup L_n^-$. Define $h_n = \psi^0_n g + \psi^+_n - \psi^-_n$. We have the following properties:

\begin{enumerate}
    \item $\|h_n\| \leq 1$ because $h_n$ is the sum of three functions of norm at most 1 with disjoint supports.
    \item $h_n-g = (\psi^0_n-1) g + \psi^+_n - \psi^-_n$ vanishes on $V'_n$ and on $K\setminus V'_{n-1}$.
    \item $\|f_\alpha + h_n\|=2$ for all $\alpha<\kappa$. To see this, pick $t\in L_n^{\varepsilon_\alpha}\cap W_\alpha^k$. Since $t\in L_n^{\varepsilon_\alpha}$ we have that $h_n(t) = \varepsilon_\alpha \psi_n^{\varepsilon_\alpha}(t) = 1$, and $$\varepsilon_\alpha (f_\alpha(t) + h_n(t)) = \varepsilon_\alpha f_\alpha(t) + \varepsilon_\alpha h_n(t) > \left(1 - \frac{1}{k}\right) + 1.$$
    This proves that $\|f_\alpha + h_n\|\geq 2-1/k$ for all $k$.
\end{enumerate}
We now take  an open set $V_n$ such that $L^0_n\subseteq \overline{V}_n\subseteq V'_n$. This completes the inductive procedure.

The functions $h_n-g$ constitute a bounded sequence of pairwise disjoint functions. They are pairwise disjoint because $h_n-g$ vanishes on $V'_n$ and on  $K\setminus V'_{n-1}$, so it is nonzero only on $V'_{n-1}\setminus V'_n$. Therefore $h_n-g$ weakly converges to 0, so $h_n$ weakly converges to $g$, so $h_n\in S$ for big enough $n$. This finishes the proof because we knew that $\|f_\alpha+h_n\|=2$ for all $n$.

At this point, the theorem is proved if the assumptions $\mathfrak{pDau}(C(K))\geq\omega$ and $\mathfrak{r}(K)\geq\omega$ hold.

What happens if $\mathfrak{r}(K)<\omega$? Then we have a finite family $\mathcal{F}$ in Definition~\ref{def:reapingnumber}. The family $\mathcal{F}_0$ of minimal nonempty intersections of sets in $\mathcal{F}$ still satisfies Definition~\ref{def:reapingnumber}, and it is a pairwise disjoint family. Some $A\in \mathcal{F}_0$ must be a singleton because, otherwise, we could take disjoint closed finite sets $L_1$ and $L_2$ that intersect all $A\in\mathcal{F}_0$. We conclude that $K$ has an isolated point $x$. Then $\mathcal{F}=\{x\}$ also serves as family, so $\mathfrak{r}(K)=1$. By Theorem~\ref{theo:CKDaugavet}, $C(K)$ fails the Daugavet property, and moreover $\mathfrak{oct}(C(K))=1$. We get
    $$1 = \mathfrak{pDau}(C(K))=\mathfrak{Dau}(C(K)) = \mathfrak{oct}(C(K))= \mathfrak{rig}\text{-}\mathfrak{oct}(C(K)) =\mathfrak{r}(
    K).$$

It is time to deal with the case when $\mathfrak{pDau}(C(K))<\omega$ and $\mathfrak{r}(K)\geq \omega$. In this case, by Theorem~\ref{theo:CKDaugavet}, $C(K)$ fails the perfect Daugavet property, $\mathfrak{pDau}(C(K))=1$, and $K$ has a $G_\delta$-point, say that $\{x\}=f^{-1}(\{0\})$ for some continuous $f\colon K\longrightarrow [0,1]$. Then we have that $\mathcal{F} = \bigl\{f^{-1}\bigl([0,1/n)\bigr) \colon n<\omega\bigr\}$ is as in Definition~\ref{def:reapingnumber}, because if $x\not\in L_1$, then $\min\{f(t) \colon t\in L_1\}>0$ and therefore $L_1\cap f^{-1}\bigl([0,1/n)\bigr)=\emptyset$ for big enough $n$. We conclude that $\mathfrak{r}(K)\leq\omega$, indeed $\mathfrak{r}(K)=\omega$ by our assumption in this last case. We observed before that if $K$ has an isolated point then $\mathfrak{r}(K) = 1$. So $K$ has no isolated points and, therefore, $C(K)$ has the Daugavet property and it is in turn octahedral. Taking into account that $\mathfrak{rig}\text{-}\mathfrak{oct}(C(K))=1$ in virtue of Remark~\ref{remark:C(K)Gdeltarigidoh}, it follows that
$$\mathfrak{pDau}(C(K))= \mathfrak{rig}\text{-}\mathfrak{oct}(C(K)) = 1 < \mathfrak{r}(
    K) = \omega \leq \mathfrak{Dau}(C(K)) \leq \mathfrak{oct}(C(K)).$$
But we proved at the beginning that $\mathfrak{oct}(C(K))\leq \mathfrak{r}(K)$ whenever $\mathfrak{oct}(C(K))\geq\omega$. So $$\mathfrak{pDau}(C(K))= \mathfrak{rig}\text{-}\mathfrak{oct}(C(K)) = 1 <\omega = \mathfrak{Dau}(C(K)) = \mathfrak{oct}(C(K))=\mathfrak{r}(
    K).$$
    
We finish the proof with the following observation. By Theorem~\ref{theo:CKDaugavet}, the fact that $K$ has a $G_\delta$-point and not an isolated point is equivalent to $\mathfrak{pDau}(C(K)) = 1 <\omega \leq \mathfrak{Dau}(C(K))$. Since such inequalities only happened in the last case that we considered, this last case must be equivalent to $K$ having $G_\delta$-points and no isolated points.
\end{proof}

\section{The reaping number of a compact space}\label{section:reapingnu}

In this section we dig deeper into the investigation of the reaping number $\mathfrak{r}(K)$, with implications in the behavior of the transfinite Daugavet properties in $C(K)$ spaces.

The first thing to say is that the reaping number of a compact space generalises the known notion of \emph{reaping number of a Boolean algebra} $\mathbb{B}$, defined as the least cardinality $\mathfrak{r}(\mathbb{B})$ of a subset $X\subseteq \mathbb{B}\setminus\{0\}$ such every element $a\in B$ either contains an element of $X$ or is disjoint from an element of $X$. Such set $X$ is called a \emph{reaping set} of $\mathbb{B}$, cf. \cite{bs1,dsw,monk}.

\begin{proposition}\label{prop:0dim}
Let $K$ be a totally disconnected compact topological space and let $\mathbb{B}$ be the Boolean algebra of the clopen subsets of $K$. Then $\mathfrak{r}(K)=\mathfrak{r}(\mathbb{B})$.
\end{proposition}

\begin{proof}
Let $X$ be a family of clopens that is reaping. Then $\mathcal{F}=X$ is also a family of nonempty open subsets of $K$ that satisfies the requirements of Definition~\ref{def:reapingnumber}, because in a totally disconnected compact space, for every two disjoint closed sets $L_1,L_2$ there is a clopen set $A$ such that $L_1\subset A$ and $L_2\cap A = \emptyset$. This proves $\mathfrak{r}(K)\leq\mathfrak{r}(\mathbb{B})$. For the converse inequality, consider $\mathcal{F}$ a family of nonempty open subsets of $K$ as in Definition~\ref{def:reapingnumber}. For each open set $W\in \mathcal{F}$ take a nonempty clopen set $A_W\subseteq W$. The family $\{A_W\}_{W\in\mathcal{F}}$ is reaping.
\end{proof}

The proof suggests that we can extend the term \emph{reaping family} for families $\mathcal{F}$ as in Definition~\ref{def:reapingnumber}. It is convenient to notice some equivalent formulations of this property.

\begin{proposition}\label{reapingfamily}
    For a family $\mathcal{F}$ of nonempty open subsets of $K$, the following are equivalent
    \begin{enumerate}
        \item for every disjoint closed sets $L_1, L_2$ there exists $W\in\mathcal{F}$  such that $W\cap L_1 = \emptyset$ or $W\cap L_2=\emptyset$.
        \item for every open sets $V_1, V_2$ with  $K=V_1\cup V_2$ there exists $W\in \mathcal{F}$ such that $W\subseteq V_1$ or $W\subseteq V_2$.
        \item no continuous function $f\in C(K)$ reaches both values 0 and 1 on all members of $\mathcal{F}$.
    \end{enumerate}
A family $\mathcal{F}$ satisfying these conditions will be called a \emph{reaping family}. The reaping number $\mathfrak{r}(K)$ is the least cardinality of reaping family in $K$.
\end{proposition}

\begin{proof}
    The equivalence of (1) and (2) is obvious considering $V_i$ as the complement of $L_i$, while (3) is equivalent to (1) by Urysohn's lemma. The last statement is just recalling that (1) is as in Definition~\ref{def:reapingnumber}.
\end{proof}

\subsection{Some notable computations of the reaping number} With Proposition~\ref{prop:0dim} and Theorem~\ref{theo:CK} at hand, we can have a look at the literature, check for what algebras $\mathbb{B}$ the reaping number $\mathfrak{r}(\mathbb{B})$ has been computed, and we will know $\mathfrak{Dau}(C(K))$ for the compact spaces $K$ with those algebras of clopens.

Let us start with the case of $\ell_\infty/c_0$. The space $\ell_\infty/c_0$ is isometric to the space of continuous functions $C(K)$ for $K=\beta\omega\setminus\omega$, which is a totally disconnected compact space whose algebra of clopens is isomorphic to $\mathcal{P}(\omega)/\mathit{fin}$, the algebra of subsets of the natural numbers modulo finite sets. The reaping number $\mathfrak{r}(\mathcal{P}(\omega)/\mathit{fin})$  is called just \emph{the reaping number} $\mathfrak{r}$. It can be described as the least cardinality of a family of infinite subsets of the natural numbers such that every other infinite set either contains a member of the family or is disjoint from a member of the family. It is a well known cardinal between $\omega_1$ and $\mathfrak{c}$, one can see \cite{blass} for a review of its behavior and relation to other cardinal characteristics of the continuum in different models of set theory. Now, with the help of Theorem~\ref{theo:CK}, we get the following result.

\begin{proposition}
    $\mathfrak{Dau}(\ell_\infty/c_0) = \mathfrak{r}$.
\end{proposition}

Other spaces for which we can use Theorem~\ref{theo:CK} to obtain the ``Daugavetian'' cardinal invariant are $L_\infty(\{0,1\}^\kappa)$. Let us give present some relevant cardinals. Here, $\mathcal{N}$ is the family of subsets of $\mathbb{R}$ of measure zero, $[\kappa]^\omega$ is the family of all countable subsets of $\kappa$, and if $\mathcal{F}$ is a family of sets, $\operatorname{cof}\mathcal{F}$ is the least cardinality of a subfamily $\mathcal{G}\subseteq\mathcal{F}$ such that every set in $\mathcal{F}$ is contained in some set in $\mathcal{G}$.

\begin{proposition}
    $\mathfrak{Dau}(L_\infty(\{0,1\}^\kappa)) = \max(\operatorname{cof}\mathcal{N},\operatorname{cof}[\kappa]^\omega)$ for every infinite cardinal $\kappa$.     
    In particular, $\mathfrak{Dau}(L_\infty[0,1]) = \operatorname{cof}\mathcal{N}$.
\end{proposition}

\begin{proof}
The space $L_\infty(\{0,1\}^\kappa)$ is isometric to a space $C(K)$ where $K$ is the Stone space of the measure algebra of $\{0,1\}^\kappa$. The reaping number (under the name \emph{weak density}) of this measure algebra was computed in \cite[Theorem 1]{burke} as the maximum between $\operatorname{cof}[\kappa]^\omega$ and the least cardinality of a family of subsets of $\mathbb{R}$ of positive measure such that every set of positive measure almost contains one from the family. This latter cardinal was later shown in \cite{ckp} to coincide with the more familiar $\operatorname{cof}\mathcal{N}$, thus giving the formula in the statement. We have that $\omega_1\leq \operatorname{cof}\mathcal{N} \leq \mathfrak{c}$ and $\kappa\leq \operatorname{cof}[\kappa]^\omega \leq \kappa^\omega$. Again, one can see \cite{blass} for information about $\operatorname{cof}\mathcal{N}$.  For the last statement in the Proposition recall that, by Maharam's theorem, $L_\infty[0,1]$ is isometric to $L_\infty(\{0,1\}^\omega)$.
\end{proof}

By Maharam's theorem, spaces $L_\infty(\mu)$ for localizable atomless measures $\mu$ can be expressed as $\ell_\infty$-direct sums of spaces $L_\infty(\{0,1\}^\kappa)$. We will discuss what happens with these sums later, but we can easily determine now that they behave well for $C(K)$ spaces. Recall that the \v{C}ech-Stone compactification of a completely regular topological space $Z$ is a compact space $\beta Z$ that contains $Z$ as a subspace, and such that every continuous function $f\colon Z\longrightarrow [a,b]$ has a unique continuous extension $\hat{f}\colon \beta Z\longrightarrow [a,b]$. 

\begin{proposition}\label{prop:linftysum-ckspaces}
    Let $\{K_i\}_{i\in I}$ be a family of compact spaces and $\beta\bigcup_{i\in I} K_i$ the \v{C}ech-Stone compactification of their discrete union. Then $\mathfrak{r}\left(\beta\bigcup_{i\in I}K_i\right) = \min\{\mathfrak{r}(K_i)\colon i\in I\}$ and, therefore, 
    $$
    \mathfrak{Dau}\left(\left(\bigoplus_{i\in I}C(K_i)\right)_{\ell_\infty}\right) = \min_{i\in I}\mathfrak{Dau}(C(K_i)).
    $$
\end{proposition}

\begin{proof}
From the definition of \v{C}ech-Stone compactification, it follows that $C(\beta\bigcup_{i\in I}K_i)$ is isometric to $(\bigoplus_{i\in I}C(K_i))_{\ell_\infty}$, and so the last formula of the statement follows from the first one and Theorem~\ref{theo:CK}. The inequality $[\leq]$ holds for any compactification of the discrete union. If $\mathcal{F}$ is a reaping family in $K_i$ for some $i$, then the same family is reaping in $\beta\bigcup_{i\in I}K_i$. For the inequality $[\geq]$ we prove that if $\mathcal{F}$ is reaping in $\beta\bigcup_{i\in I}K_i$ then there exists $i$ such that $\{A\cap K_i \colon A\in \mathcal{F}, A\cap K_i\neq\emptyset\}$ is reaping in $K_i$. If it is not the case, then for every $i$ there is a continuous function $f_i\colon K_i\longrightarrow [0,1]$ that reaches both values 0 and 1 on $A\cap K_i$ whenever $A\in\mathcal{F}$ and $A\cap K_i\neq\emptyset$. By the definition of \v{C}ech-Stone compactification, we have a continuous function $f\colon K\longrightarrow [0,1]$ such that $f|_{K_i}=f_i$, and this function contradicts that the family $\mathcal{F}$ was reaping in $\beta\bigcup_{i\in I}K_i$.
\end{proof}

As a corollary, we get that $L_\infty(\mu)$ has the $\omega$-perfect Daugavet property for every localizable measure $\mu$ since, in this case, Maharam's decomposition for $L_1(\mu)$ (see \eqref{eq:Maharam} below) produces the analogous one for $L_\infty(\mu)$ by duality, that is, 
\begin{equation*}
L_\infty(\mu)\equiv \Big(\bigoplus_{\lambda\in \Lambda}L_\infty(\{-1,1\}^{\kappa_\lambda})\Big)_{\ell_\infty},
\end{equation*}
where $\kappa_\lambda\geq  \omega$ for every $\lambda\in \Lambda$.

\begin{corollary}
Let $\mu$ be a localizable atomless measure. Then, $L_\infty(\mu)$ has the $\omega$-perfect Daugavet property.   
\end{corollary}

Another example that we can find in the literature is that the reaping number of the free Boolean algebra on $\kappa$ generators is $\kappa$ \cite[Example 12]{monk}, which translates into the fact that $\mathfrak{r}(\{0,1\}^\kappa) = \kappa$. We prove a more general fact about how the reaping number behaves in products. First, an easy observation. Remember that a \emph{$\pi$-basis} of a topological space $T$ is a family $\mathcal{B}$ of nonempty open subsets of $T$ such that every nonempty open set contains a member of $\mathcal{B}$.

\begin{lemma}\label{pibasis}
If $\mathcal{B}$ is a $\pi$-basis of a compact topological space $K$, then $\mathfrak{r}(K)$ is the minimal cardinality of a reaping family $\mathcal{F}\subseteq \mathcal{B}$. 
\end{lemma}

\begin{proof}
    Just observe that if $\mathcal{F}$ is a reaping family and for every $A\in \mathcal{F}$ we pick $B_A\in \mathcal{B}$ such that $B_A\subseteq A$, then the family $\mathcal{F}' =\{B_A \colon A\in\mathcal{F}\}\subseteq \mathcal{B}$ is also reaping, and $|\mathcal{F}'|\leq |\mathcal{F}|$.
\end{proof}

We are now able to calculate the reaping number of a product space.

\begin{proposition}\label{reapingofproduct}
    Let $\{K_i\}_{i<\kappa}$ be an infinite family of $\kappa$ many compact spaces with more than one point, and $K=\prod_{i<\kappa}K_i$. Then $$\mathfrak{r}(K) \geq \kappa \vee \sup_{i<\kappa}\mathfrak{r}(K_i).$$
\end{proposition}

\begin{proof}
We can consider the $\pi$-basis $\mathcal{B}$ of $K$ consisting of the products $\prod_{i<\kappa}W_i$ of open sets where $W_i=K_i$ for all but finitely many $i$. First, we fix $i<\kappa$ and prove that $\mathfrak{r}(K_i)\leq \mathfrak{r}(K)$. Let $\mathcal{F}\subset \mathcal{B}$ be a family of cardinality $\mathfrak{r}(K)$ such that no $f\in C(K)$ reaches 0 and 1 on all members of $\mathcal{F}$. Consider $\mathcal{F}_i$ the family of all the $i$-th coordinate factors of members of $\mathcal{F}$. It is enough to check that no $g\in C(K_i)$ reaches 0 and 1 on all members of $\mathcal{F}_i$. But this is clearly true because otherwise the composition of $g$ with the $i$-th projection would reach 0 and 1 on all members of $\mathcal{F}$. Now we prove that $\kappa\leq \mathfrak{r}(K)$. For this, we need to show that no reaping family $\mathcal{F}\subseteq \mathcal{B}$ for $K$ has size less than $\kappa$. Indeed, if we assume by contradiction that there exists a reaping family of cardinality less than $\kappa$, by a cardinality argument there must exists a coordinate $i$ such that the $i$-th coordinate of all members of $\mathcal{F}$ is $K_i$. Consider any continuous function $g\in C(K_i)$ that reaches 0 and 1. The composition of $g$ with the $i$-th projection gives $f\in C(K)$ that reaches 0 and 1 on all members of $\mathcal{F}$, a contradiction with the fact that $\mathcal F$ is a reaping family. 
\end{proof}

For finite products, the above proof shows that $$\mathfrak{r}(K_1\times\cdots\times K_n) \geq \mathfrak{r}(K_1)\vee\cdots \vee \mathfrak{r}(K_n).$$ We do not know if the converse inequality holds here or in Proposition~\ref{reapingofproduct}. We make some remarks on this problem later, see Proposition~\ref{prop:product+}.

A consequence of Proposition~\ref{reapingofproduct} and Theorem~\ref{theo:CK} is the following result.

\begin{corollary}
    If $\kappa$ is an infinite cardinal and $K$ is a product of $\kappa$ many compact spaces with at least two points, then $\mathfrak{Dau}(C(K))\geq \kappa$.
\end{corollary}

\begin{remark}\label{remark:C(K)injlargdaugavet}
Another Boolean algebra for which the reaping number is computed in the literature is the completion of the free Boolean algebra on $\kappa$ generators, that has reaping number $\kappa$ \cite[Proposition 26]{monk}. This algebra was a key tool in the study of transfinite octahedrality properties in \cite{amcrz23}. Its associated compact space by Stone duality $K$ satisfies that $\mathfrak r(K)=\kappa$ and is extremally disconnected (the closure of every open set is open, also called Stonean), so $C(K)$ is 1-injective, see \cite[Theorem 2.5.11]{dualCK}, and $\mathfrak{pDau}(C(K))=\mathfrak r(K)=\kappa$ according to Theorem~\ref{theo:CK}.
\end{remark}

Concerning the above remark we can say even more, in fact, the associated compact space by Stone duality of the completion of the free Boolean algebra on $\kappa$ generators is the \emph{Gleason cover} of $\{0,1\}^\kappa$. The Gleason cover of a compact space $K$ is a compact space $G(K)$ that is characterised by being extremally disconnected and by the existence of an irreducible continuous surjection $f\colon G(K)\longrightarrow K$, cf.\ \cite{walker}. Recall that a continuous surjection $f\colon L\longrightarrow K$ is irreducible if the restriction of $f$ to any proper closed subspace of $L$ is never surjective onto $K$. This is equivalent to saying that for every nonempty open set $A\subseteq L$ there is a nonempty open set $B\subseteq K$ such that $f^{-1}(B)\subseteq A$. An observation that is useful to estimate the reaping number of Gleason covers is the following:

\begin{proposition}\label{irreduciblereduction}
    If $f\colon L\longrightarrow K$ is an irreducible surjection, then $\mathfrak{r}(K)\leq \mathfrak{r}(L)$.
\end{proposition}

\begin{proof}
    If $\{A_i\}_{i<\mathfrak{r}(L)}$ is a reaping family for $L$, and we choose nonempty open sets $B_i\subseteq f(A_i)$, then $\{B_i\}_{i<\mathfrak{r}(L)}$ is reaping for $K$. This is because if $g\in C(K)$ reached 0 and 1 on every $B_i$, then $g\circ f$ would reach 0 and 1 on every $A_i$. We conclude that $\mathfrak{r}(K)\leq \mathfrak{r}(L)$. 
\end{proof}

\begin{proposition}\label{dualremark} Let  $K$ be the Gleason cover of a perfect compact metric space $M$. Then $C(K)$ is a 1-injective Banach space  isomorphic to $\ell_\infty$ that satisfies $\mathfrak{Dau}(C(K)) = \mathfrak{pDau}(C(K)) = \omega.$
\end{proposition}

\begin{proof}
Let $f\colon K\longrightarrow M$ be the irreducible surjection. The family of those sets of the form $f^{-1}(W)$ for all basic open sets $W$ of $M$, constitutes a countable $\pi$-basis of $K$. It follows from this that $K$ has no isolated points and $\mathfrak{r}(K) = \omega$. Moreover, picking a point in each set of the countable $\pi$-basis, one gets that $K$ is separable, which gives that $C(K)$ is isomorphic to $\ell_\infty$, by~\cite[Corollary 6.9.5]{dualCK}.  For the equality $\mathfrak{pDau}(C(K))=\omega$  we have to show that $K$ has no $G_\delta$-points, and this is because in an extremally disconnected compact space $G_\delta$-points are isolated. Here is an elementary proof: Suppose that $x$ is a non-isolated $G_\delta$-point and we have open neighborhoods $\{W_n\}_{n<\omega}$ of $x$ with $\overline{W_{n+1}}\subsetneq W_n$ and $\bigcap_{n<\omega}W_n = \{x\}$. By compactness, the $W_n$ form a basis of neighborhoods of $x$. Define $V = \bigcup_{n<\omega}W_{2n}\setminus \overline{W_{2n+1}}$. Then, since $x\in \overline{V}$, the definition of extremally disconnected implies that there exists $m$ such that $W_m\subseteq \overline{V}$, which contradicts the fact that $W_{2m+1}\setminus \overline{W_{2m+2}}$ is contained in $W_m$ but disjoint from $V$.
\end{proof}

We do not know if there is any Banach space \emph{isometric} to a dual space for which $\mathfrak{Dau}(X)=\omega$ (with the Daugavet property but not the $\omega$-Daugavet property).

Taking the one-point compactification of a discrete union of $\kappa$ copies of $K$ in Proposition~\ref{dualremark} we can get spaces with $\mathfrak{Dau}
(C(K)) = \mathfrak{pDau}
(C(K)) = \omega$ of arbitrarily high density character $\kappa\geq \mathfrak{c}$ of $C(K)$. If we consider  $K$ the one-point compactification of the discrete union of $\kappa$ many copies of the unit interval, then $K$ has no isolated points but has $G_\delta$-points, so that gives examples of $$\mathfrak{pDau}(C(K)) = 1 <\omega = \mathfrak{Dau}(C(K))$$  and arbitrary uncountable density character $\kappa$ of $C(K)$.

\subsection{Other cardinal invariants related to the reaping number.}

Theorem~\ref{theo:CKDaugavet} shows that the non-transfinite Daugavet and perfect Daugavet properties of $C(K)$ are characterised by the lack of points with a small local basis. Do we have a similar result for the transfinite case? We are going to see now that there is something going on along this line, but things become more complicated. The key notion is the following:

\begin{definition}
    Let $K$ be a topological space and $x\in K$. A $\pi$-basis of the point $x$ is a family $\mathcal{B}$ of nonempty open subsets of $K$ such that every neighborhood of $x$ contains a member of $\mathcal{B}$. The minimal cardinality of a $\pi$-basis of $x$ is called $\pi_\chi(x)$, the $\pi$-character of $x$.
\end{definition}

The difference with the more popular notion of basis of neighborhoods of $x$ is that the elements of the $\pi$-basis $\mathcal{B}$ are not required to contain the point $x$. The following fact noticed by Balcar and Simon \cite{bs1} for $K$ totally disconnected holds for general compact spaces:

\begin{proposition}\label{prop:picharacter}
For every compact space $K$, we have $$\mathfrak{r}(K) \leq \min\{\pi_\chi(x) \colon  x\in K \}.$$
\end{proposition}

\begin{proof}
    It is obvious that a $\pi$-basis of any point $x$ is also a reaping family in $K$. If we use Proposition~\ref{reapingfamily}(2), either $V_1$ or $V_2$ will be a neighborhood of $x$.
\end{proof}

If that inequality could be reversed, this would be the characterisation we were dreaming of in terms of local behavior. However, Balcar and Simon \cite{bs2} gave an example of a totally disconnected compact space for which $\mathfrak{r}(K) < \min\{\pi_\chi(x) \colon  x\in K \}$. The equality has been proved only in some special cases like extremally disconnected or Stone spaces of homogeneous Boolean algebras \cite{bs1,dsw}. However, the following remarkable result from \cite{dsw} says that, in the totally disconnected case, these two cardinals cannot differ very much.

\begin{theorem}[\mbox{Dow, Stepr\={a}ns, Watson, \cite{dsw}}]\label{theo:dsw}
    For every totally disconnected compact space $K$, we have $$\mathfrak{r}(K) \leq \min\{\pi_\chi(x) \colon  x\in K \}
     \leq \mathfrak{r}(K)^+.$$
    In other words, either $\min\{\pi_\chi(x) \colon  x\in K \} = \mathfrak{r}(K)$ or $\min\{\pi_\chi(x) \colon  x\in K \}=\mathfrak{r}(K)^+$.
\end{theorem}

Combining Proposition~\ref{prop:picharacter} and Theorem~\ref{theo:dsw} with Theorem~\ref{theo:CK}, we can state:

\begin{corollary}\label{cor:picharacter}
    Let $K$ be a compact space and $\kappa$ an infinite cardinal. \begin{enumerate} \item If $C(K)$ has the $\kappa$-Daugavet property, then no point of $K$ has a $\pi$-basis of cardinality $\leq \kappa$. \item If $K$ is totally disconnected and $C(K)$ fails the $\kappa$-Daugavet property, then there is a point in $K$ with a $\pi$-basis of cardinality $\leq \kappa^+$.
    \end{enumerate}    
\end{corollary}

Later discussion, cf.\ Proposition~\ref{twostepscardinalsnew}(1), suggests the following shorter notation for the minimal $\pi$-character of a compact space: $$\tilde{\mathfrak{r}}_\omega(K) = \min\{\pi_\chi(x) \colon x\in K\}.$$
This cardinal invariant often behaves in a more friendly way than the reaping number, which combined with Theorem~\ref{theo:dsw} can sometimes tell that reaping number behaves \emph{nicely up to a successor} on totally disconnected compact spaces. For example, we are only able to prove the inequality $\mathfrak{r}(K\times L)\geq \mathfrak{r}(K)\vee \mathfrak{r}(L)$, but it is a trivial exercise that $\tilde{\mathfrak{r}}_\omega(K\times L) = \tilde{\mathfrak{r}}_\omega(K)\vee \tilde{\mathfrak{r}}_\omega(L)$ for any compact spaces $K$, $L$. From this, we conclude:

\begin{proposition}\label{prop:product+}
If $K$ and $L$ are totally disconnected compact spaces, then $$\mathfrak{r}(K)\vee \mathfrak{r}(L)\leq \mathfrak{r}(K\times L) \leq \mathfrak{r}(K)^+\vee \mathfrak{r}(L)^+.$$
\end{proposition}

\begin{proof} Just concatenate some inequalities,
    \[\mathfrak{r}(K)\vee \mathfrak{r}(L) \leq \mathfrak{r}(K\times L) \leq \mathfrak{r}_\omega(K\times L) = \mathfrak{r}_\omega(K)\vee\mathfrak{r}_\omega(L) \leq \mathfrak{r}(K)^+\vee \mathfrak{r}(L)^+.\qedhere\]
\end{proof}

Since $C(K\times L)$ is isometric to the injective tensor product $C(K)\iten C(L)$, we have the following corollary combining Proposition~\ref{prop:product+} and Theorem~\ref{theo:CK}:

\begin{corollary}
    Let $K$ and $L$ be compact spaces and $\kappa$ an infinite cardinal. 
    \begin{enumerate}
     \item If either $C(K)$ or $C(L)$ has the $\kappa$-Daugavet property, then $C(K)\iten C(L)$ has the $\kappa$-Daugavet property.
    \item Assume that $K$ and $L$ are totally disconnected. If $C(K)\iten C(L)$ has the $\kappa$-Daugavet property, then either $C(K)$ or $C(L)$ must have the $\kappa^+$-Daugavet property.
    \end{enumerate}
\end{corollary}

In the classical setting, the Daugavet property of $C(K)\iten C(L)$ is equivalent to either $C(K)$ or $C(L)$ having the Daugavet property. And the same holds for the perfect Daugavet property. This is an application of  Theorem~\ref{theo:CKDaugavet}, because it is clear that $K\times L$ has an isolated point (respectively a $G_\delta$-point) if and only if both $K$ and $L$ have an isolated point (respectively a $G_\delta$-point).

We do not know if Theorem~\ref{theo:dsw} and Corollary~\ref{cor:picharacter}(2) hold without the assumption of $K$ being totally disconnected. We went carefully through the proof of Theorem~\ref{theo:dsw} in \cite{dsw} in the hope of a straightforward generalisation but we found an essential difficulty, that we are going to explain now. The key is the consideration of the following auxiliary cardinals.

\begin{definition}
    For a Boolean algebra $\mathbb{B}$ and $n\geq k$ natural numbers, the $n,k$-reaping number $\mathfrak{r}_{n,k}(\mathbb{B})$ is defined as the least cardinality of a family $\mathcal{A}$ of nonzero elements of $\mathbb{B}$ such that for every $n$-partition $p_1,\ldots,p_n\in \mathbb{B}$ there exists $a\in \mathcal{A}$ that intersects less than $k$ many of the $p_i$'s.
\end{definition}

These cardinals generalise the reaping number in the sense that $\mathfrak{r}(\mathbb{B}) = \mathfrak{r}_{2,2}(\mathbb{B})$. A result of Laflamme~\cite{laflamme} says that $\mathfrak{r}_{n,k}(K) = \mathfrak{r}_{\lceil \frac{n}{k-1}\rceil,2}(K)$ which allows to reduce the study cardinals $\mathfrak{r}_{n,2}$, 
$$\mathfrak{r}(\mathbb{B}) = \mathfrak{r}_{2,2}(\mathbb{B}) \leq \mathfrak{r}_{3,2}(\mathbb{B})\leq \mathfrak{r}_{4,2}(\mathbb{B}) \leq \cdots$$

The proof of Theorem~\ref{theo:dsw} follows from putting together two facts:

\begin{lemma}\label{twostepscardinals} Let $K$ be a totally disconnected compact space and $\mathbb{B}$ the algebra of clopen sets of $K$.
\begin{enumerate}
    \item $\sup_{n<\omega}\mathfrak{r}_{n,2}(\mathbb{B}) = \min\{\pi_\chi(x) \colon  x\in K\}$, cf. \cite[Proposition 1.6]{bs1}, \cite[Proposition 1.1]{dsw}. 
    \item $\mathfrak{r}_{n,2}(\mathbb{B})\leq \mathfrak{r}(\mathbb{B})^+$ for every $n<\omega$, cf. \cite[Corollary 3.2]{dsw}.
\end{enumerate}
\end{lemma}

When trying to export this scheme to general compact spaces, there are two different ways in which we can define these auxiliary cardinals.

\begin{definition}
	For a compact space $K$ and integers $n\geq k$ the cover-reaping number $\tilde{\mathfrak{r}}_{n,k}(K)$ is the least cardinal of a family $\mathcal{F}$ of nonempty open subsets of $K$ such that for every open cover of $K$ with $n$ many open sets, there exists $W\in \mathcal{F}$ that is contained in the union of less that $k$ many open sets from the cover.
\end{definition}

\begin{definition}
	For a compact space $K$ and integers $n\geq k$ the separation-reaping number $\ddot{\mathfrak{r}}_{n,k}(K)$ is the least cardinal of a family $\mathcal{F}$ of nonempty open subsets of $K$ such that for every pairwise disjoint closed subsets $L_1,\ldots,L_n$ of $K$ there exists $W\in\mathcal{F}$ that has nonempty intersection with less than $k$ many of the $L_i$'s.
\end{definition}

When $n=k=2$, Proposition~\ref{reapingfamily} means that $\ddot{\mathfrak{r}}_{2,2}(K) = \tilde{\mathfrak{r}}_{2,2}(K) = \mathfrak{r}(K)$. We do not have any example where $\ddot{\mathfrak{r}}_{n,k}(K) \neq \tilde{\mathfrak{r}}_{n,k}(K)$, we can only prove the obvious inequality $\ddot{\mathfrak{r}}_{n,k}(K) \leq \tilde{\mathfrak{r}}_{n,k}(K)$. By Tietze's extension theorem, the condition (3) of Proposition~\ref{reapingfamily} corresponds to the separation-reaping number.

\begin{proposition}
	$\ddot{\mathfrak{r}}_{n,k}(K)$ is the least cardinality of a family $\mathcal{F}$ of nonempty open subsets of $K$ such that for every continuous function $f\colon K\longrightarrow [1,n]$ there exists $W\in \mathcal{F}$ on which $f$ takes less than $k$ many integer values.
\end{proposition}

Both variants of reaping numbers generalise the original Boolean definitions.

\begin{proposition}
	If $K$ is totally disconnected and $\mathbb{B}$ is the algebra of the clopen subsets of $K$, then $\tilde{\mathfrak{r}}_{n,k}(K) = \ddot{\mathfrak{r}}_{n,k}(K) = {\mathfrak{r}}_{n,k}(\mathbb{B})$.
\end{proposition}

\begin{proof}
The definition of ${\mathfrak{r}}_{n,k}(\mathbb{B})$ is the same as the definition we gave for $\ddot{\mathfrak{r}}_{n,k}(K)$ except that $L_1,\ldots,L_n$ is a clopen partition, rather than just pairwise disjoint closed sets. But when $K$ is totally disconnected that makes no difference because if our family $\mathcal{F}$ has that property for clopen partitions and we take $L_1,\ldots,L_n$ pairwise disjoint closed sets, we can find a clopen partition $L'_1,\ldots,L'_n$ with $L_i\subseteq L'_i$. So indeed $\ddot{\mathfrak{r}}_{n,k}(K) = {\mathfrak{r}}_{n,k}(\mathbb{B})$. We focus now on $\tilde{\mathfrak{r}}_{n,k}(K)$. The cover in its definition can be supposed to be a partition into clopen sets. This is because if property stated for $\mathcal{F}$ holds for such partitions, and $K=\bigcup_{i=1}^n V_i$ is an open cover, then we can find a clopen partition $K = \bigcup_{i=1}^n\tilde{V}_i$ with $\tilde{V}_i \subseteq V_i$ (this is an exercise in topology, choose a clopen neighborhood of every point of $K$ contained in some $V_i$, then use compactness to find a finite subcover and then group the clopen pieces into blocks contained in each $V_i$). Also, the family $\mathcal{F}$ can be supposed to be made of clopen sets. Thus, calling $\mathcal{P}_n$ the set of all $n$-partitions of $\mathbb{B}$,
\begin{align*}
\tilde{\mathfrak{r}}_{n,k}(K) &= \min\left\{|\mathcal{A}| \colon  \mathcal{A}\subset \mathbb{B}\setminus\{0\}, \forall \{p_1,\ldots,p_n\}\in\mathcal{P}_n \ \exists a\in\mathcal{A}\ |\{i \colon  a\cap p_i\neq \emptyset\}|<k \right\} \\ & = \min\left\{|\mathcal{A}| \colon  \mathcal{A}\subset \mathbb{B}\setminus\{0\}, \not\exists \{p_1,\ldots,p_n\}\in\mathcal{P}_n \ \forall a\in\mathcal{A}\ |\{i \colon  a\cap p_i\neq \emptyset\}|\geq k \right\}=\mathfrak{r}_{n,k}(\mathbb{B}).\qedhere
\end{align*}
\end{proof}

We have the following picture

$$\begin{array}{ccccccccc}
   \mathfrak{r}(K) & \leq & \tilde{\mathfrak{r}}_{3,2}(K) & \leq & \tilde{\mathfrak{r}}_{4,2}(K) &  \leq & \tilde{\mathfrak{r}}_{5,2}(K) & \leq & \cdots \\
   \mathrel{\rotatebox{90}{$=$}} & & \mathrel{\rotatebox{90}{$\leq$}} & & \mathrel{\rotatebox{90}{$\leq$}} & & \mathrel{\rotatebox{90}{$\leq$}} & &\\
     \mathfrak{r}(K)  & \leq & \ddot{\mathfrak{r}}_{3,2}(K)  & \leq & \ddot{\mathfrak{r}}_{4,2}(K) & \leq & \ddot{\mathfrak{r}}_{5,2}(K) & \leq & \cdots
\end{array}$$

When we analysed the proofs of the two statements of Lemma~\ref{twostepscardinals}, one of them seems to work fine only for cover-reaping numbers, and the other only for separation-reaping numbers.
Unfortunately, the two statements of Proposition~\ref{twostepscardinalsnew} below do not match to get the desired conclusion, because we do not know if the equality $\ddot{r}_{n,k}(K) = \tilde{r}_{n,k}(K)$ holds in general.
\begin{proposition}\label{twostepscardinalsnew} Let $K$ be a compact space.
\begin{enumerate}
    \item $\sup\limits_{n<\omega}\tilde{\mathfrak{r}}_{n,2}(K) = \min\{\pi_\chi(x) \colon x\in K\}$.
    \item $\ddot{\mathfrak{r}}_{n,2}(K)\leq \mathfrak{r}(K)^+$ for every $n<\omega$.
\end{enumerate}
\end{proposition}

\begin{proof}
	For $[\leq$] in (1), observe that if $\mathcal{F}$ a $\pi$-basis of neighborhoods of some $x$, then $\mathcal{F}$ also satisfies the property in the definition of $\tilde{\mathfrak{r}}_{n,2}$ just by considering the member of the cover that contains $x$. For $[\geq]$ in (1),  consider $\mathcal{F}_n$ a family that witnesses the definition of $\tilde{\mathfrak{r}}_{n,2}$ for every $n$. We prove that $\mathcal{F}:=\bigcup_n \mathcal{F}_n$ is the $\pi$-basis of some $x\in K$. For this, consider $\mathcal{Z}$ the family of all closed subsets of $K$ that have nonempty intersection with every element of $\mathcal{F}$. We claim that $\mathcal{Z}$ has the finite intersection property. This is because if we have $Z_1,\ldots,Z_n\in\mathcal{Z}$ with $\bigcap_{i=1}^n Z_i = \emptyset$, then the sets $K\setminus Z_i$ form an open cover, thus we would have $W\in\mathcal{F}_n$ such that $W\subset K\setminus Z_i$ for some $i$, which contradicts that $Z_i$ cuts every $W\in\mathcal{F}$. We conclude that there exists $x\in\bigcap_{Z\in\mathcal{Z}} Z$. We have that $\mathcal{F}$ is a $\pi$-basis for $x$ because if $V$ is an open neighborhood of $x$, then $x\not\in K\setminus V$, so $K\setminus V\not\in\mathcal{Z}$, so there exists $W\in\mathcal{F}$ disjoint from $K\setminus V$. We found $W\in\mathcal{F}$ with $W\subseteq V$.

Following \cite{dsw}, for natural numbers $n,m,i,j,k,q$, the statement
$$\binom{m}{n} \not\rightarrow \binom{j}{i}^{1,1}_{k,q} $$
means that there is a function $h$ that takes $k$ values on a product set $X\times Y$ of size $n\times m$ such that $h$ takes more than $q$ values on every product set $A\times B$ of size $i\times j$. The following generalizes \cite[Lemma 3.3]{dsw} for general compact spaces.

\emph{Claim}. 
	Suppose that $$\binom{m}{n} \not\rightarrow \binom{j}{i}^{1,1}_{k,q}. $$ Then any compact space that satisfies $\ddot{\mathfrak{r}}_{n,i}(K) > \ddot{\mathfrak{r}}_{k,q+1}(K)$ also satisfies $\ddot{\mathfrak{r}}_{m,j}(K) \leq \ddot{\mathfrak{r}}_{k,q+1}(K)^+$.

    \emph{Proof of the claim.} Let $h\colon \{1,\ldots,n\}\times \{1,\ldots,m\}\longrightarrow k$ be a witness of the hypothesis. By Tietze's theorem, we can find a continuous function $h'\colon [1,n]\times [1,m]\longrightarrow [1,k]$ such that $h'(t,s) = h(\xi,\zeta)$ whenever $\xi\in\{1,\ldots,n\}$, $\zeta\in\{1,\ldots,m\}$, $|t-\xi|\leq 1/4$, $|s-\zeta|\leq 1/4$ . Let $K$ be a compact space such that $\ddot{\mathfrak{r}}_{n,i}(K) > \ddot{\mathfrak{r}}_{k,q+1}(K)$, and consider the cardinal $\kappa= \ddot{\mathfrak{r}}_{k,q+1}(K)^+$. Let $\mathcal{A}$ be a family of nonempty open subsets of $K$ of cardinality $\ddot{\mathfrak{r}}_{k,q+1}(K)$ such that every $f\in C(K,[1,\ldots,k])$ takes at most $q$ integer values on some $W\in\mathcal{A}$. Without loss of generality we can assume that $\mathcal{A}$ consists of cozero open sets. We are going to recursively define families $\mathcal{B}_\alpha$ for $\alpha\leq \kappa$.
	\begin{itemize}
		\item $\mathcal{B}_0=\mathcal{A}$.
		\item For $\gamma\leq \kappa$ limit ordinal, define $\mathcal{B}_\gamma = \bigcup_{\beta<\gamma}\mathcal{B}_\beta$.
		\item Given $\alpha<\kappa$, we will be given $\mathcal{B}_\alpha$ of cardinality $\ddot{\mathfrak{r}}_{k,q+1}(K)<\ddot{\mathfrak{r}}_{n,i}$. We will be thus able to find $f_\alpha\in C(K,[1,n])$ that takes at least $i$ integer values on every $W\in\mathcal{B}_\alpha$. We define $\mathcal{B}_{\alpha+1}$ as the set of all nonempty finite intersections of the form
		$$ W \cap f_{\alpha_1}^{-1}(I_1-1/4,I_1+1/4)\cap \cdots\cap f_{\alpha_\eta}^{-1}(I_\eta-1/4,I_\eta+1/4)$$
		with $W\in\mathcal{A}$, $\eta\in\{1,\ldots,i\}$,  $I_1,\ldots,I_\eta\in\{1,\ldots,n\}$, $\alpha_1<\cdots<\alpha_\eta\leq \alpha$.	\end{itemize}
	Finally, assuming that $\ddot{\mathfrak{r}}_{m,j} >\kappa$ we choose a function $f_\kappa\in C(K,[1,m])$ that takes at least $j$ integer values on every $W\in\mathcal{B}_\kappa$. For every $\alpha<\kappa$ we define a function $c_\alpha\in C(K,[1,k])$ as $c_\alpha = h'\circ (f_\alpha\times f_\kappa)$. By the definition of $\mathcal{A}$, for every $\alpha<\kappa$ there exists $W_\alpha\in\mathcal{A}$ such that $c_\alpha$ takes at most $q$ integer values on $W_\alpha$. By a cardinality argument, we can find $\alpha_1<\cdots<\alpha_i<\kappa$ for which we  $W_{\alpha_1}=\cdots=W_{\alpha_i}=:W$ and moreover $c_{\alpha_1},\ldots,c_{\alpha_i}$ attain exactly the same set $E$ of at most $q$ many integers on $W$.
	
	Now, we choose inductively $I_1,\ldots,I_i\in \{1,\ldots,n\}$ such that
	\begin{itemize}
		\item $I_\eta\neq I_{\eta'}$ if $\eta\neq \eta'$,
		\item There is a point $x\in W$ such that $|f_{\alpha_\eta}(x)-I_\eta|<1/4$ for $\eta=1,\ldots,i$. 
	\end{itemize}	
	This is done inductively as follows. Suppose that $I_1,\ldots,I_\eta$ have been chosen as above up to some $\eta<i$ and we define $I_{\eta+1}$. We assume inductively that there exists $x_\eta\in W$ with $|f_{\alpha_{\eta'}}(x_\eta)-I_{\eta'}|<1/4$ for $\eta'=1,\ldots,\eta$. In other words
	$$ W \cap f_{\alpha_1}^{-1}(I_1-1/4,I_1+1/4)\cap \cdots\cap f_{\alpha_\eta}^{-1}(I_\eta-1/4,I_\eta+1/4)\neq\emptyset.$$
	The open set above is a member of the family $\mathcal{B}_{\alpha_{\eta+1}}$, so by definition of $f_{\alpha_{\eta+1}}$ we know that $f_{\alpha_{\eta+1}}$ takes at least $i$ many integer values on that set. We take $I_{\eta+1}$ one such integer value different from the $I_{\eta'}$ with $\eta'\leq\eta$. This finishes the construction of the $I_\eta$'s.
	
	Now, we have that $$ W \cap f_{\alpha_1}^{-1}(I_1-1/4,I_1+1/4)\cap \cdots\cap f_{\alpha_i}^{-1}(I_i-1/4,I_i+1/4)$$
	is a nonempty member of $\mathcal{B}_\kappa$. Therefore $f_\kappa$ takes at least $j$ integer values $f_\kappa(y_1)=J_1,\ldots,f_\kappa(y_j) = J_j$ on that set. We prove now that $h$ takes only $q$ many values on $\{I_1,\ldots,I_i\}\times \{J_1,\ldots,J_j\}$ which contradicts the choice of $h$. To see this, fix $\eta\leq i$ and $\xi\leq j$ and we will have that
	\[h(I_\eta,J_\xi)  = h'(f_{\alpha_\eta}(y_\xi), f_\kappa(y_\xi)) = c_{\alpha_\eta}(y_\xi)\in E.\] 
This finishes the proof of the claim.

Proposition~\ref{twostepscardinalsnew}(2)  is now \cite[Theorem 3.1]{dsw}. Since $\ddot{\mathfrak{r}}_{2,2}(K)\leq \ddot{\mathfrak{r}}_{3,2}(K)\leq \cdots$, it is enough to prove the case when $n$ is prime. It is easy to see that
	$$\ddot{\mathfrak{r}}_{n,2}(K)\geq \ddot{\mathfrak{r}}_{n,3}(K)\geq \cdots \geq \ddot{\mathfrak{r}}_{n,n-1}(K)\geq \ddot{\mathfrak{r}}_{n,n}(K) \leq {\mathfrak{r}}(K) $$
	Either  $\ddot{\mathfrak{r}}_{n,2}(K)\leq {\mathfrak{r}}(K)$ and we are done, or otherwise we can pick $q$ the greatest integer $q<n$ for which $\ddot{\mathfrak{r}}_{n,q}(K)>{\mathfrak{r}}(K)$. As noticed in \cite[Lemma 3.2]{dsw}, when $n$ is prime, the function $h(u,v) = u+v \text{ mod }n$ shows that $$\binom{n}{n} \not\rightarrow \binom{2}{q}^{1,1}_{n,q},$$
	and therefore we can apply the claim: If $\ddot{\mathfrak{r}}_{n,q}(K)>\ddot{\mathfrak{r}}_{n,q+1}(K)$ (which is in fact true by the maximality of $q$) then $\ddot{\mathfrak{r}}_{n,2}(K)\leq \ddot{\mathfrak{r}}_{n,q+1}(K)^+$. But the maximality of $q$ also implies that $\ddot{\mathfrak{r}}_{n,q+1}(K) \leq {\mathfrak{r}}(K)$, which concludes the proof.
\end{proof}

\begin{remark}\label{remark Laflamme new}The proof of Laflamme's theorem mentioned above seems to work fine for both cover-reaping and separation-reaping numbers, so $\tilde{\mathfrak{r}}_{n,k}(K) = \tilde{\mathfrak{r}}_{\lceil \frac{n}{k-1}\rceil,2}(K)$ and $\ddot{\mathfrak{r}}_{n,k}(K) = \ddot{\mathfrak{r}}_{\lceil \frac{n}{k-1}\rceil,2}(K)$ for all compact spaces $K$. We will not reproduce those arguments here, the only step of the proof that requires an adaptation is the  implication $(\Leftarrow)$ of \cite[Theorem 2.1]{dsw}. Lemma~\ref{rnn} was just a particular case of this fact, $\ddot{\mathfrak{r}}_{3,3}(K) = \ddot{\mathfrak{r}}_{2,2}(K)$.
\end{remark}

\section{Direct sums}\label{sect:directsums} In this section we pursue to study the stability result of transfinite Daugavet properties by taking $\ell_1$- and $\ell_\infty$-sums. As a consequence of this analysis we will derive in the next section a description of those $L_1(\mu)$ spaces having the $\kappa$-(perfect) Daugavet property.

Let us start with a characterisation of when an arbitrary $\ell_1$-sum of Banach spaces enjoy the transfinite Daugavet properties.

\begin{proposition}\label{prop: k-Daugavet in ell_1 sums}
    Let $\{X_i\colon i\in I\}$ be a family of Banach spaces and $\kappa$ a cardinal. Then
\begin{enumerate}
    \item $\left(\bigoplus_{i\in I}X_i\right)_{\ell_1}$ has the $\kappa$-Daugavet property if and only if $X_i$ has the $\kappa$-Daugavet property for every $i\in I$.
    \item $\left(\bigoplus_{i\in I}X_i\right)_{\ell_1}$ has the $\kappa$-perfect Daugavet property if and only if $X_i$ has the $\kappa$-perfect Daugavet property for every $i\in I$.
\end{enumerate}
In other words, $\mathfrak{Dau}\left(\left(\bigoplus_{i\in I} X_i\right)_{\ell_1}\right) = \min\limits_{i\in I}\{\mathfrak{Dau}(X_i)\}\ $ and $\ \mathfrak{pDau}\left(\left(\bigoplus_{i\in I} X_i\right)_{\ell_1}\right) = \min\limits_{i\in I}\{\mathfrak{pDau}(X_i)\}$.
\end{proposition}

\begin{proof}We will focus on the proof of (1), the proof of (2) is similar. For $[\Rightarrow]$, we assume that the direct sum has the $\kappa$-Daugavet property, we fix $i_0\in I$, and we show that $X_{i_0}$ has the $\kappa$-Daugavet property. Let $\{y_\alpha\colon \alpha<\kappa\}\subseteq S_{X_{i_0}}$, let $S(B_{X_{i_0}}, x^*_{i_0}, \alpha)$, and $\varepsilon>0$. Pick $\delta>0$ such that $3\delta<\min\{\alpha, \varepsilon\}$.
    
    Consider the set $\{(0,\dots,0,y_\alpha,0,\dots)\colon \alpha<\kappa\}$ in $S_{(\bigoplus_{i\in I}X_i)_{\ell_1}}$ and the slice 
    \[
    S(B_{(\bigoplus_{i\in I}X_i)_{\ell_1}}, (0,\dots,0, x^*_{i_0},0,\dots), \delta).
    \]
    Since $(\bigoplus_{i\in I}X_i)_{\ell_1}$ has the $\kappa$-Daugavet property, there is a $(u_i)\in S_{(\bigoplus_{i\in I}X_i)_{\ell_1}}$ such that

    \begin{equation}\label{eq: 1-sum orthogonal}
      \Vert (0,\dots,0,y_\alpha,0,\dots)+(u_i)\Vert=\Vert y_\alpha+u_{i_0}\Vert + \sum_{i\in I\setminus\{i_0\}}\Vert u_i\Vert >2-\delta  
    \end{equation}
      holds for every $\alpha<\kappa$,
    and
    \begin{equation}\label{eq: 1-sum slice}
           x^*_{i_0}(u_{i_0})>1-\delta.
    \end{equation}
The equation \eqref{eq: 1-sum slice} implies that $\Vert u_{i_0}\Vert>1-\delta$. Therefore,  $\sum_{i\in I\setminus\{i_0\}}\Vert u_n\Vert<\delta$ because $\sum_{i\in I}\Vert u_i\Vert=1$. Hence, from \eqref{eq: 1-sum orthogonal}, we deduce that 
\[
\Vert y_\alpha+u_{i_0}\Vert>2-2\delta.
\]
holds for every $\alpha<\kappa$. Finally, set $w:=\frac{u_{i_0}}{\Vert u_{i_0}\Vert}$, then
\[
\Vert y_\alpha+w\Vert\geq 2-2\delta-\Vert u_{i_0}-w\Vert=2-2\delta-\Big|1-\Vert u_{i_0}\Vert\Big|>2-3\delta>2-\varepsilon
\]
for every $\alpha<\kappa$.

    $[\Leftarrow]$. We assume now that all $X_i$ have the $\kappa$-Daugavet property and we prove it for $(\bigoplus_{i\in I}X_i)_{\ell_1}$. Indeed, fix $\{(y^i_\alpha)\colon \alpha<\kappa, i\in I\}$ in $S_{(\bigoplus_{i\in I}X_i)_{\ell_1}}$, a slice $S(B_{(\bigoplus_{i\in I}X_i)_{\ell_1}}, (x^*_i), \eta)$, and $\varepsilon>0$. Let $\delta>0$ be such that $2\delta<\min\{\eta,\varepsilon\}$. We will prove the claim in two steps.

    \textit{Step 1.} Assume first that $I$ is countable, say $I\equiv\N$ (if $I$ is finite, the result follows from the infinite countable case and the $[\Rightarrow]$ part already proved). 
    Since every $X_n$ has the $\kappa$-Daugavet property, then, for every $n$, by Lemma~\ref{lemma: subspace Daugavet}, we can find $u_n\in S_{X_n}$ such that
\begin{equation}\label{eq: 1-sum slice2}
    x^*_n(u_n)\geq (1-\delta)\Vert x^*_n\Vert
\end{equation}
and
\begin{equation}\label{eq: 1-sum orthogonal2}
    \Vert y^n_\alpha+ru_n\Vert>(1-\delta)(\Vert y^n_\alpha\Vert+r)
\end{equation}
for every $\alpha<\kappa$ and $r\geq0$. (In case that $x_n^*=0$, one may forget Eq.~\eqref{eq: 1-sum slice2}).

Pick a sequence $(r_n)\in [0,1]$ such that 
\begin{equation}\label{eq: duality of ell_1 and ell_infty}
  (r_n u_n)\in S_{\ell_1(X_n)}\quad \text{ and }\quad \sum_{n\geq 1}r_n\Vert x^*_n\Vert>1-\delta.  
\end{equation}
 Now, by \eqref{eq: 1-sum slice2} and \eqref{eq: duality of ell_1 and ell_infty}, we have that
\[
\Big\langle (x^*_n), (r_nu_n)\Big\rangle=\sum_{n\geq 1}r_nx^*_n(u_n)\geq (1-\delta)\sum_{n\geq 1}r_n\Vert x^*_n\Vert>(1-\delta)^2>1-\eta.
\]
Hence, $(r_n u_n)\in S(B_{\ell_1(X_n)}, (x^*_1,x^*_2,\dots), \eta)$. Finally, by \eqref{eq: 1-sum orthogonal2} and \eqref{eq: duality of ell_1 and ell_infty}, we have that
\begin{align*}
    \Vert (y^n_\alpha)+(r_nu_n)\Vert &=\sum_{n\geq 1}\Vert y^n_\alpha+r_nu_n\Vert 
    > (1-\delta)\sum_{n\geq 1}(\Vert y^n_\alpha\Vert+ r_n)
    =2(1-\delta)>2-\varepsilon.
\end{align*}

\textit{Step 2.} Assume now that $I$ is an arbitrary infinite set. Find $(x_i)\in S_{(\bigoplus_{i\in I}X_i)_{\ell_1}}$ such that $\Big\langle (x^*_i), (x_i)\Big\rangle>1-\delta$. Denote by $J:=\{i\in I\colon x_i\neq 0\}$ and observe that $J$ is at most countable. Hence, $\sum_{i\in J}x^*_i(x_i)>1-\delta$ and $\sup_{i\in J}\|x^*_i\|>1-\delta$.

By Step 1, we can find a sequence $(u_i)_{i\in J}\in S_{(\bigoplus_{i\in J}X_i)_{\ell_1}}$ such that 
\begin{equation}\label{eq: step 1 slice}
  \sum_{i\in J}x^*_i(u_i)>(1-\delta)^2
\end{equation}
and
\begin{equation}\label{eq: step 1 orthogonal}
   \|(y^i_\alpha)_{i\in J} +(u_i)_{i\in J} \|\geq (1-\delta)(\|(y^i_\alpha)_{i\in J}\|+1)  
\end{equation}
for every $\alpha<\kappa$.

For every $i\in I$, let
\[
w_i=
\begin{cases}
u_i,\text{ if } i\in J\\
0, \text{ if } i\in I\setminus J.
\end{cases}
\]
Then it is clear that $(w_i)\in S_{(\bigoplus_{i\in I}X_i)_{\ell_1}}$. Now, by \eqref{eq: step 1 slice}, we have that
\[
\Big\langle (x^*_i), (w_i)\Big\rangle=\sum_{i\in J}x^*_i(u_i)>(1-\delta)^2>1-\eta.
\]
Hence, $(w_i)\in S(B_{(\bigoplus_{i\in I}X_i)_{\ell_1}}, (x^*_i), \eta)$. Finally, by \eqref{eq: step 1 orthogonal}, we have that
\begin{align*}
    \Vert (y^i_\alpha)+(w_i)\Vert &=\|(y^i_\alpha)_{i\in J} +(u_i)_{i\in J} \|+\sum_{i\in I\setminus J}\Vert y^i_\alpha\Vert \\
    &> (1-\delta)(\|(y^i_\alpha)_{i\in J}\|+1)+\sum_{i\in I\setminus J}\Vert y^i_\alpha\Vert \\
    &\geq (1-\delta)\Big(1+\sum_{i\in J}\Vert y^i_\alpha\Vert+\sum_{i\in I\setminus J}\Vert y^i_\alpha\Vert \Big)=2(1-\delta)>2-\varepsilon.\qedhere
\end{align*}
\end{proof}

Next we turn to study the transfinite Daugavet properties for $\ell_\infty$-sums of Banach spaces. To begin with, let us start with a necessary condition in order to an $\ell_\infty$-sum enjoy the transfinite Daugavet property.

\begin{proposition}\label{prop:necelinfisum}
Let $\{X_i\colon  i\in I\}$ be an arbitrary family of Banach spaces, and $\kappa$ a cardinal.
\begin{enumerate}
    \item If $\left(\bigoplus_{i\in I}X_i\right)_{\ell_\infty}$ has the $\kappa$-Daugavet property, then $X_i$ has the $\kappa$-Daugavet property for every $i\in I$.
    \item If $\left(\bigoplus_{i\in I}X_i\right)_{\ell_\infty}$ has the $\kappa$-perfect Daugavet property, then $X_i$ has the $\kappa$-perfect Daugavet property for every $i\in I$.
\end{enumerate}
In other words, $\mathfrak{Dau}\left(\left(\bigoplus_{i\in I} X_i\right)_{\ell_\infty}\right)\leq \min\limits_{i\in I}\mathfrak{Dau}(X_i)\ $ and $\ \mathfrak{pDau}\left(\left(\bigoplus_{i\in I} X_i\right)_{\ell_\infty}\right)\leq \min\limits_{i\in I}\mathfrak{pDau}(X_i)$. 
\end{proposition}

\begin{proof}
Since both proofs are similar, let us prove (2), being the proof of (1) completely similar.

So assume that $\left(\bigoplus_{i\in I}X_i\right)_{\ell_\infty}$ has the $\kappa$-perfect Daugavet property, fix $i_0\in I$, and let us prove that $X_{i_0}$ has the $\kappa$-perfect Daugavet property. In order to do so, suppose that $Y = \{y^\alpha \colon  \alpha<\kappa\}$ is a subset of the sphere of $X_{i_0}$ of cardinality less than $\kappa$ and take a slice $S=S(B_{X_{i_0}},f,\alpha)$ for some $f\in S_{X_{i_0}^*}$ and $\alpha>0$. For every $\alpha<\kappa$, let $\tilde{y}^\alpha$ be the vector of $\left(\bigoplus_{i\in I}X_i\right)_{\ell_\infty}$ that has $y^\alpha$ on the $i_0$-th coordinate and zero elsewhere. Moreover, define the functional $F\colon \left(\bigoplus_{i\in I}X_i\right)_{\ell_\infty}\longrightarrow \mathbb R$ such that
$$F((x_i)):=f(x_{i_0})\ \forall (x_i)\in \ell_\infty(X_i),$$
which is a norm-one element of $\ell_\infty(X_i)^*$. Since $\left(\bigoplus_{i\in I}X_i\right)_{\ell_\infty}$ has the $\kappa$-perfect Daugavet property we can obtain a vector $x$ with $\|x\|=1$, $F(x)>1-\alpha$ and $\|x-\tilde{y}^\alpha\|=2$ for all $\alpha$. On the one hand, observe that
$$f(x_{i_0})=F(x)>1-\alpha,$$
so $x_{i_0}\in S$. On the other hand, notice that $\|x_i-\tilde{y}_i^\alpha\| = \|x_i\|\leq 1$ for $i\neq i_0$, so we must have $\|x_{i_0}-\tilde{y}_{i_0}^\alpha\| = \|x_{i_0}-y^\alpha\| =2$ for all $\alpha$. So $x_{i_0}$ is the vector of $X_{i_0}$ that shows that it has the $\kappa$- perfect Daugavet property.
\end{proof}

A converse can be established for finite $\ell_\infty$-sums.

\begin{proposition}\label{prop:sufic0sum}
Let $\{X_i\colon  i\in I\}$ be a finite family of Banach spaces, and $\kappa$ a cardinal.
\begin{enumerate}
    \item $\left(\bigoplus_{i\in I}X_i\right)_{\ell_\infty}$ has the $\kappa$-Daugavet property if and only if $X_i$ has the $\kappa$-Daugavet property for every $i\in I$.
    \item $\left(\bigoplus_{i\in I}X_i\right)_{\ell_\infty}$ has the $\kappa$-perfect Daugavet property if and only if $X_i$ has the $\kappa$-perfect Daugavet property for every $i\in I$.
\end{enumerate}
In other words, $\mathfrak{Dau}\left(\left(\bigoplus_{i\in I}X_i\right)_{\ell_\infty}\right)= \min\limits_{i\in I}\mathfrak{Dau}(X_i)$ $\ $ and $\ $ $\mathfrak{pDau}\left(\left(\bigoplus_{i\in I}X_i\right)_{\ell_\infty}\right)= \min\limits_{i\in I}\mathfrak{pDau}(X_i)$.
\end{proposition}

\begin{proof}
Since the proofs of (1) and (2) are similar, we will prove with no loss of generality (1), being the proof of (2) completely similar. Moreover, the necessary condition is proved in Proposition~\ref{prop:necelinfisum}, so let us prove the sufficient condition. In order to do so, take $S=\{s^\alpha\colon  \alpha<\kappa\}$ a set of cardinality at most $\kappa$ inside the unit sphere of $X:=(\bigoplus_{i\in I}X_i)_{\ell_\infty}$ and $W$ be a non-empty weak open subset of $B_{X}$. Since the weak topology of $B_X$ is the product topology of weak topologies of $B_{X_i}$ we can assume up to taking a smaller open set that $W=\prod_{i\in I} W_i$, where every $W_i$ is a non-empty weakly open subset of $B_{X_i}$. Fix $\varepsilon>0$. Since $\|s^\alpha\|=1$, for every $\alpha$ there exists $i[\alpha]\in I$ such that $\|s^\alpha_{i[\alpha]}\|=1$. For every $j\in I$, let $S_j = \{s^\alpha_j\colon  i[\alpha]=j \}$. Each $S_j$ is a subset of the sphere of $X_j$ of cardinality at most $\kappa$. Therefore, since every $X_i$ has the $\kappa$-Daugavet property, we can find $x_j$ in $W_j$ such that $\|x_j-s^\alpha_j\| > 2-\varepsilon$ whenever $j=i[\alpha]$. Now, consider $x= (x_j)_{j\in I}$, which is clearly an element of $\prod\limits_{i\in I} W_i=W$. For every $\alpha<\kappa$, we will have
\begin{align*}
\|x-s^\alpha\|&\geq \left\|x_{i[\alpha]}-s^\alpha_{i[\alpha]}\right\| 
> 2-\varepsilon.
\end{align*}
Consequently, we get that $\left(\bigoplus_{i\in I}X_i\right)_{\ell_\infty}$ has the $\kappa$-Daugavet property in virtue of  Lemma~\ref{lemma: Daugavet and CCS}(iii).
\end{proof}

We do not know whether the above result holds true for arbitrary sums of Banach spaces (see Question~\ref{ques:linftysum}). On the other hand, the result for infinite $c_0$-sums is false.

\begin{remark}\label{remark:c0sumfalso}
Let $\{X_i\colon i\in I\}$ an infinite family of non-trivial Banach spaces and let $X:=\left(\bigoplus_{i\in I}X_i\right)_{c_0}$. Then $X$ fails to be $\omega_1$-octahedral.

Indeed, take an infinite countable set $J\subseteq I$ and consider, for every $j\in I$, a vector $x_j\in S_X$ such that $x_j(i)=0$ if $i\neq j$. Now, given any $x\in S_X$, since $J$ is infinite, there exists some $j_0\in J$ such that $\Vert x(j)\Vert<\frac{1}{2}$. It is immediate $\Vert x\pm x_j\Vert\leq \frac{3}{2}$. Since $x\in S_X$ was arbitrary we conclude the failure of $\omega_1$ octahedrality.
\end{remark}

\section{Spaces \texorpdfstring{$L_1(\mu)$}{L1(mu)}}\label{sect:L1} 
Our aim here is to study the transfinite Daugavet property for $L_1(\mu)$ spaces. To do so, we will study the case of the spaces $L_1(\{0,1\}^\kappa)$ and provide the general result from Maharam's decomposition (see \eqref{eq:Maharam}). 

So we begin with the proof of the case of $L_1(\{0,1\}^\kappa)$, whose proof is inspired by the Step 1 in \cite[Proposition 6.11]{brss26}

\begin{theorem}\label{theo: L1 has k-Daugavet}
$\mathfrak{Dau}(L_1(\{0,1\}^\kappa) = \kappa$ for every infinite cardinal $\kappa$.
\end{theorem}

\begin{proof}
Since $L_1(\{0,1\}^\kappa)$ has a dense subset of cardinality $\kappa$, this space does not have the $\kappa$-Daugavet property. If $\kappa=\omega$, then $\mathfrak{Dau}(L_1(\{0,1\}^\omega) = \omega$ as the measure space is atomless. So we may and do suppose that $\kappa>\omega$ and we have to prove that the $\kappa'$-Daugavet property holds for every infinite cardinal $\kappa'<\kappa$. Indeed, fix $\varepsilon>0$, $\Phi\in L_\infty\left(\{-1,1\}^\kappa\right)$ with $\|\Phi\|=1$, and a subspace $\mathcal{Z}$ of $L_1\left(\{-1,1\}^\kappa\right)$ with $\operatorname{dens}(\mathcal{Z})=\kappa'<\kappa$. It is enough to find $g \in L_1\left(\{-1,1\}^\kappa\right)$ with $\|g\|=1$,
\[
\|f+g\| \geqslant(1-\varepsilon)(\|f\|+1) \quad \text { for all } f \in \mathcal{Z} .
\]
and
\[
\int_{\{-1,1\}^\kappa} \Phi g\, d \mu>1-\varepsilon.
\]
Since every function in $L_1\left(\{-1,1\}^\kappa\right)$ depends on countably many coordinates and $\operatorname{dens}(\mathcal{Z})< \kappa$ and $\kappa>\omega$, there exists a subset $\Lambda$ of $\kappa$ with $|\Lambda|<\kappa$ such that all functions in $\mathcal{Z}$ only depend on the coordinates from $\Lambda$. Fix $n \in \mathbb{N}$ large enough that $2 \cdot 2^{-n} \leqslant \varepsilon$ and choose ordinals $\alpha_1, \ldots, \alpha_n \in \kappa \backslash \Lambda$. For each sign $\sigma \in\{-1,1\}^n$ consider the set
\[
A_\sigma:=\left\{x \in\{-1,1\}^\kappa\colon  x\left(\alpha_j\right)=\sigma(j),\, j=1, \ldots, n\right\}.
\]
This family of sets form a partition of $\{-1,1\}^\kappa$ into sets of measure $2^{-n}$. Notice that, for a function $f \in \mathcal{Z}$,

\[
\int_{A_\sigma}|f| d \mu=2^{-n} \int_{\{-1,1\}^\kappa}|f| d \mu
\]
since $f$ does not depend on the coordinates $\alpha_1, \ldots, \alpha_n$; in other words, $\left\|f \cdot \chi_{A_\sigma}\right\|=2^{-n}\|f\|$. Since $\mu$ is non-atomic and $\|\Phi\|=1$, we can find a set $B$ such that $\mu(B)<\varepsilon$ and $|\Phi(w)|>1-\varepsilon$ for every $w\in B$. 
As the sets $A_\sigma$ form a partition, we can find a sign $\sigma$ such that $\mu(A_\sigma\cap B)>0$. Observe that $\left\|f \cdot \chi_{B\cap A_\sigma}\right\|\leq 2^{-n}\|f\|$.  We can now define the function $g$ we are after. Define $g:=\frac{\operatorname{sign}(\Phi)\cdot\chi_{B\cap A_\sigma}}{\mu(B\cap A_\sigma)}$; note that $\|g\|=1$. Then, for every $f \in \mathcal{Z}$, we have:

\[
\begin{aligned}
\|f+g\| & \geqslant\left\|f \cdot \chi_{(B\cap A_\sigma)^{\mathrm{C}}}+g\right\|-\left\|f \cdot \chi_{B\cap A_\sigma}\right\| \\
& \geq\left\|f \cdot \chi_{(B\cap A_\sigma)^{\mathrm{C}}}\right\|+1-2^{-n}\|f\| \\
& \geq \left(1-2^{-n}\right)\|f\|+1-2^{-n}\|f\| \\
& \geqslant\left(1-2 \cdot 2^{-n}\right) \cdot(\|f\|+1) \geqslant(1-\varepsilon) \cdot(\|f\|+1)
\end{aligned}
\]
and
\[
\int_{\{-1,1\}^\kappa} \Phi g\, d \mu=\frac{1}{\mu(B\cap A_\sigma)}\int_{B\cap A_\sigma} \operatorname{sign}(\Phi)\cdot \Phi \, d \mu>1-\varepsilon.\qedhere
\]
\end{proof}

For atomless $L_1(\mu)$ spaces, Maharam's theorem produces a decomposition of $L_1(\mu)$ as follows:
\begin{equation}\label{eq:Maharam}
L_1(\mu)\equiv \Big(\bigoplus_{\lambda\in \Lambda}L_1(\{-1,1\}^{\kappa_\lambda})\Big)_{\ell_1},
\end{equation}
where $\kappa_\lambda\geq  \omega$ for every $\lambda$ (see e.g.\  \cite[Section 14, Theorem 9]{lacey74} and \cite[Remark on page 501]{DefantFloret}). Now, Theorem~\ref{theo: L1 has k-Daugavet} in combination with Proposition~\ref{prop: k-Daugavet in ell_1 sums} gives the computation of $\mathfrak{Dau}(L_1(\mu))$ for arbitrary measures.

\begin{corollary}
Let $(\Omega,\Sigma,\mu)$ be an atomless measure space and consider the  Maharam's decomposition 
$$
L_1(\mu)\equiv \Big(\bigoplus_{\lambda\in \Lambda}L_1(\{-1,1\}^{\kappa_\lambda})\Big)_{\ell_1}.
$$
Then,
$$
\mathfrak{Dau}(L_1(\mu)) = \min\{\kappa_\lambda\colon \lambda\in \Lambda\}
$$ 
\end{corollary}

For the perfect Daugavet property, the spaces $L_1(\mu)$ have a dramatically different behaviour. 

\begin{proposition}\label{L1failsperfect}
Let $(\Omega,\Sigma,\mu)$ be a measurable space. Then, $\mathfrak{pDau}(L_1(\mu))=1$.
\end{proposition}

We will deduce the result from an easy general fact that has its own interest.

\begin{lemma}\label{lemma:smoothpoint-noperfect}
Let $X$ be a Banach space with the perfect Daugavet property. Then, the norm of $X$ is not smooth at any point.
\end{lemma}

\begin{proof}
Suppose otherwise that $x\in S_X$ is a smooth point of the norm and that $f\in S_{X^*}$ is the unique support functional for $x$. Consider the slice $S:=\{z\in B_X\colon f(z)<0\}$ and find $y\in S$ such that $\|x+y\|=2$. If $g\in S_{X^*}$ satisfies that $g(x+y)=2$, then $g(x)=g(y)=1$, so $g=f$ and $y\notin S$, a contradiction.
\end{proof}

Observe that a byproduct of the above result is that no separable Banach space may have the perfect Daugavet property. Indeed, the norm of every separable Banach space is smooth at a dense subset (Mazur's theorem, see \cite[Theorem 8.2]{FHHMZ2011}, for instance).

\begin{corollary}
If $X$ is a Banach space with the perfect Daugavet property, then $X$ is not separable.
\end{corollary}

\begin{proof}[Proof of Proposition~\ref{L1failsperfect}]
The starting point is that $\mu$ has to be atomless in order that $L_1(\mu)$ has the Daugavet property. Let us first suppose that $\mu$ is an atomless probability space. Then, it is immediate that $\chi_\Omega$ is a smooth point of $L_1(\mu)$ and hence, Lemma~\ref{lemma:smoothpoint-noperfect} gives us that $L_1(\mu)$ fails the perfect DPr. In the general case, thanks to the decomposition given in \eqref{eq:Maharam}, there is an atomless probability measure $\nu$ such that $L_1(\mu)\equiv L_1(\nu)\oplus_1 Z$ for convenient closed subspace $Z$. Hence, $L_1(\mu)$ fails the perfect DPr by Proposition~\ref{prop: k-Daugavet in ell_1 sums} and the previous case.
\end{proof}

We acknowledge that Samir Hamad (private communication) has independently proved Proposition~\ref{L1failsperfect} for $\sigma$-finite measures with a similar approach. On the other hand, it is shown in the very recent preprint \cite{KAASIK2026} that no separable Lipschitz-free space has the perfect Daugavet property with an argument similar than the one in Lemma~\ref{lemma:smoothpoint-noperfect}.

Let us comment that, in contrast to the situation with $C(K)$ spaces where the $\kappa$-Daugavet and $\kappa$-perfect Daugavet are equivalent for any infinite $\kappa$, the space $L_1(\{-1,1\}^\tau)$ is an example where the $\kappa$-Daugavet property holds for all $\kappa<\tau$ (Theorem~\ref{theo: L1 has k-Daugavet}), while the $\kappa$-perfect Daugavet property fails for all $\kappa$, even for $\kappa=1$ (Proposition~\ref{L1failsperfect}). Besides, this example also shows that the perfect Daugavet property does not behave as the usual Daugavet property with respect to the duality: for every atomless localisable measure $\mu$, the space $L_\infty(\mu)$ has the perfect Daugavet property while its predual $L_1(\mu)$ fails it.

\section{Almost isometric ideals and universal disposition}\label{sect:idealsandUD}

In this section we face the problem of inheritance of the transfinite Daugavet properties by certain subspaces in a given Banach space. We borrow the following definition from \cite[Definition 4.1]{mrz25}.

\begin{definition}\label{defi:kappaideal}Let $X$ be a Banach space, $Y$ be a subspace of $X$ and $\kappa$ be an infinite cardinal. We say that $Y$ is a $\kappa$ (almost) isometric ideal in $X$ if given any subspace $E$ of $X$ with $\dens(E) <\kappa$ (and $\eps >0$), there exists an ($\eps$-)isometry $T\colon  E \longrightarrow Y$ that preserves the points of $E \cap Y$.
\end{definition}

The above notion is a generalisation of the notion of \textit{almost isometric ideal} from \cite[Definition 1.3]{aln2}, which in the above language is nothing but $\omega$ almost isometric ideals. 

Our motivation for considering the notion of $\kappa$ (almost) isometric ideals is double. On the one hand, in \cite[Proposition 4.14]{mrz25} it is proved that $\kappa$ almost isometric ideals (respectively, $\kappa$ isometric ideals) inherit the $\kappa'$ octahedrality (respectively, the $\kappa'$-rigid octahedrality) from the ambient space for every $\kappa'<\kappa$. On the other hand, in \cite[Proposition 3.8]{aln2} it is proved that the Daugavet property is inherited by almost isometric ideals. 

Motivated by the above results we get the following proposition.

\begin{proposition}\label{prop:hereaiidealargeDauga}
Let $X$ be a Banach space, let $Y$ be a subspace of $X$, and let $\kappa,\kappa'$ cardinals with $\kappa'<\kappa$ and $\kappa$ infinite. Then:
\begin{enumerate}
    \item If $X$ has the $\kappa'$-Daugavet property and $Y$ is a $\kappa$ almost isometric ideal, then $Y$ has the $\kappa'$-Daugavet property.
    \item If $X$ has the $\kappa'$-perfect Daugavet property and $Y$ is a $\kappa$ isometric ideal in $X$, then $Y$ has the $\kappa'$-perfect Daugavet property.
\end{enumerate}
\end{proposition}

\begin{proof}
Let us prove first (1). In order to do so, let $\{y_\alpha\colon \alpha<\kappa'\}\subseteq S_Y$, $v_0\in B_Y$, and $\varepsilon>0$, and let us find 
\begin{equation}\label{eq:proofalmostisometricideals}
z\in \conv\bigl(\{y\in B_Y\colon  \Vert y_\alpha-y\Vert>2-\varepsilon\ \  \forall \alpha\}\bigr)
\end{equation}
with $\Vert v_0-z\Vert<\varepsilon$. This is enough in virtue of Proposition~\ref{prop:charlargeDPconvexhull}. 

Since $\{y_\alpha\colon  \alpha<\kappa'\}\subseteq S_Y\subseteq S_X$ and $v_0\in B_Y\subseteq B_X$, as $X$ has the $\kappa'$-Daugavet property, we can find an element 
$$\sum_{i=1}^n \lambda_i x_i\in \conv\bigl(\{x\in B_X\colon  \Vert y_\alpha-x\Vert>2-\tfrac{\varepsilon}{3} \ \ \forall \alpha<\kappa'\}\bigr)
$$
such that $\Vert v_0-\sum_{i=1}^n \lambda_i x_i\Vert<\tfrac{\varepsilon}{3}$.

Define $E:=\spn(\{y_\alpha\colon \alpha<\kappa'\}\cup\{x_1,\ldots, x_n\}\cup\{v_0\})\subseteq X$. Since $\dens(E)\leq \kappa'<\kappa$ and $Y$ is a $\kappa$-almost isometric ideal in $X$, we can find a bounded linear operator $T\colon E\longrightarrow Y$ with the following properties:
\begin{enumerate}
    \item $T(e)=e$ for every $e\in E\cap Y$ (in particular, $T(y_0)=y_0$ and $T(y_\alpha)=y_\alpha$ holds for every $\alpha<\kappa'$).
    \item $(1-\delta)\Vert e\Vert \leq \Vert T(e)\Vert\leq (1+\delta)\Vert e\Vert$ for every $e\in Y$,
\end{enumerate}
where $\delta>0$ is taken small enough to guarantee that $(1-\delta)(2-\frac{\varepsilon}{3})-\delta-\frac{\varepsilon}{3}>2-\varepsilon$ and that $(1+\delta)\frac{\varepsilon}{3}+\frac{\varepsilon}{3}+\delta<\varepsilon$.

Define $v_i:=\frac{T(x_i)}{\Vert T(x_i)\Vert}\in S_Y$ for every $1\leq i\leq n$. Then, on the one hand, given $1\leq i\leq n$ we get, using condition (2) on $T$ and the fact that $x_i\in \{x\in B_X\colon  \Vert y_\alpha-x\Vert>2-\frac{\varepsilon}{3}\ \forall \alpha<\kappa'\}$ that, given $\alpha<\kappa'$, we get
\begin{align*}
\Vert y_\alpha-v_i\Vert &=\left\Vert y_\alpha-\frac{T(x_i)}{\Vert T(x_i)\Vert}\right\Vert  \geq \Vert y_\alpha-T(x_i)\Vert-\left\Vert T(x_i)-\frac{T(x_i)}{\Vert T(x_i)\Vert}\right\Vert \\& =\Vert T(y_\alpha)-T(x_i)\Vert-\frac{\vert 1-\Vert T(x_i)\Vert\vert}{\Vert T(x_i)\Vert}\Vert T(x_i)\Vert\ = \Vert T(y_\alpha-x_i)\Vert-\vert 1-\Vert T(x_i)\Vert\vert\\
& \geq (1-\delta)\Vert y_\alpha-x_i\Vert -\Vert \Vert x_i\Vert-\Vert T(x_i)\Vert\vert-\varepsilon >(1-\delta)\left(2-\frac{\varepsilon}{3}\right)-\delta-\frac{\varepsilon}{3}>2-\varepsilon.
\end{align*}
This proves that $\sum_{i=1}^n \lambda_i v_i\in \conv(\{y\in B_Y\colon  \Vert y_\alpha-y\Vert>2-\varepsilon\ \forall \alpha\})$. Now, it is time to prove that $\left\Vert v_0-\sum_{i=1}^n \lambda_i v_i\right\Vert<\varepsilon$. Indeed, we get
\begin{align*}
  \left\Vert v_0-\sum_{i=1}^n \lambda_i v_i\right\Vert &=\left\Vert v_0-\sum_{i=1}^n \lambda_i \frac{T(x_i)}{\Vert T(x_i)\Vert}\right\Vert \leq \left\Vert v_0-\sum_{i=1}^n \lambda_i T(x_i) \right\Vert+\sum_{i=1}^n \lambda_i \left\Vert T(x_i)-\frac{T(x_i)}{\Vert T(x_i)\Vert} \right\Vert   \\
  & \leq \left\Vert T(v_0)-\sum_{i=1}^n \lambda_i T(x_i)\right\Vert +\sum_{i=1}^n \lambda_i\vert 1-\Vert T(x_i)\Vert\vert  <\left\Vert T\left (v_0-\sum_{i=1}^n \lambda_i x_i\right)\right\Vert+\frac{\varepsilon}{3}+\delta \\ &\leq (1+\delta)\left\Vert v_0-\sum_{i=1}^n \lambda_i x_i\right\Vert +\frac{\varepsilon}{3}+\delta <(1+\delta)\frac{\varepsilon}{3}+\frac{\varepsilon}{3}+\delta<\varepsilon.
\end{align*}
The arbitrariness of $\varepsilon>0$ concludes the proof.

For the proof of (2), we follow almost the same arguments, but the perfect variation kills $\eps$ in \eqref{eq:proofalmostisometricideals} and, as $Y$ is a $\kappa$ isometric ideal, we can now take $T$ to be an isometry, so in the definition of the $v_i$'s, there is no need to normalise the image of the operator and the estimates are even easier. The rest of the proof is completely analogous.
\end{proof}

As a consequence, with the language presented in the introduction, we get the following corollary.

\begin{corollary}
Let $X$ be a Banach space, let $Y$ be a subspace of $X$, and let $\kappa$ be an infinite cardinal. Then:
\begin{enumerate}
    \item If $\mathfrak{Dau}(X)\geq \kappa$ and $Y$ is a $\kappa$ almost isometric ideal, then $\mathfrak{Dau}(Y)\geq \kappa$.
    \item If $\mathfrak{pDau}(X)\geq \kappa$ and $Y$ is a $\kappa$ isometric ideal, then $\mathfrak{pDau}(Y)\geq \kappa$.
\end{enumerate}
\end{corollary}

Our aim is now to take advantage of the above result in order to get examples of spaces with the $\kappa$- (perfect) Daugavet property among the class of Banach spaces of universal disposition (respectively, of almost universal disposition). Once again, our motivation is \cite[Corollary 4.5]{aln2}, where it is proved that Gurari\u{\i} spaces enjoy the Daugavet property by making use that Gurari\u{\i} spaces are almost isometric ideals in every superspace \cite[Theorem 4.3]{aln2}.

Recall that given an infinite cardinal $\kappa$, a Banach space $X$ is said to be \textit{of universal disposition} (respectively, \textit{of almost universal disposition}) \emph{for density $\kappa$} ($\text{(A)UD}_{<\kappa}$ in short) if for any two normed spaces $Y\subseteq Z$ with $\dens(Z)< \kappa$ and any isometry $t\colon Y\longrightarrow X$ there exists (for every $\varepsilon>0$) an extension $T\colon Z\longrightarrow X$ which is an isometry (respectively, such that $(1-\varepsilon)\Vert z\Vert\le \Vert T(z)\Vert\le (1+\varepsilon)\Vert z\Vert$ holds for every $z\in Z$).

Observe that for $\kappa=\omega$ we recover the strong Gurari\u{\i} and Gurari\u{\i} spaces, respectively. We refer the reader to the survey \cite{gk11} for background.

Observe also that a Banach space $X$ is a (A)UD$_{<\kappa}$ if, and only if, $X$ is a $\kappa$ (almost) isometric ideal in every Banach space that contains it \cite[Theorem 4.2]{mrz25}.

Our connection between spaces of universal disposition and transfinite Daugavet properties is described in the following result.

\begin{theorem}\label{theo:universaldisdaugalarge}
Let $X$ be a Banach space and let $\kappa$ be an infinite cardinal. Then:
\begin{enumerate}
    \item If $X$ is a UD$_{<\kappa}$ space, then $\mathfrak{pDau}(X)\geq \kappa$.
    \item Let $X$ be an AUD$_{<\kappa}$ space, then $\mathfrak{Dau}(X)\geq \kappa$.
\end{enumerate}
\end{theorem}

\begin{proof}
Let us prove (1), being the proof of (2) completely similar. Let $\kappa'$ be a cardinal with $\kappa'<\kappa$. Call $L:=(B_{X^*},w^*)$, and observe that we can embed $X$ in $C(L)$ isometrically by $\Phi\colon X\longrightarrow C(L)$ by the equation
$$\Phi(x)(x^*):=x^*(x).$$
Define $K:=\{-1,1\}^{\kappa^+}$, which satisfies that $\mathfrak{r}(K)\geq\kappa^+>\kappa$ by Proposition~\ref{reapingofproduct}. Consequently, $\mathfrak{r}(K\times L)\geq \kappa$ too. Now, consider the isometric embedding $\Psi\colon C(L)\longrightarrow C(K\times L)$ by the equation
$$\Psi(f)(k,l):=f(l)\ \ k\in K,\, l\in L.$$
By the composition of $\Phi$ and $\Psi$, we observe that $C(K\times L)$ contains $X$ isometrically. Now, one the one hand, $C(K\times L)$ has the $\kappa'$-perfect Daugavet property by Theorem~\ref{theo:CK} since $\mathfrak{r}(K\times L)\geq \kappa$. On the other hand, $X$ is a $\kappa$ isometric ideal in $C(K\times L)$ because $X$ is an almost isometric ideal in every Banach space containing it isometrically \cite[Theorem 4.2]{mrz25}. Applying (2) in Proposition~\ref{prop:hereaiidealargeDauga}, we conclude that $X$ inherits the $\kappa'$-perfect Daugavet property, so $\mathfrak{pDau}(X)\geq \kappa'$. Since $\kappa'<\kappa$ was arbitrary we conclude $\mathfrak{pDau}(X)\geq \kappa$.

The proof of (2) is exactly the same but using this time that $X$ is a $\kappa$ almost isometric ideal in $C(K\times L)$ and then, applying (2) in Proposition~\ref{prop:hereaiidealargeDauga}.
\end{proof}

\section{Tensor products}\label{sect:tensor}

The projective tensor product of $X$ and $Y$, denoted by $X \pten Y$, is the completion of the algebraic tensor product $X \otimes Y$ endowed with the norm
$$
\|z\|_{\pi} := \inf \left\{ \sum_{n=1}^k \|x_n\| \|y_n\|\colon  z = \sum_{n=1}^k x_n \otimes y_n \right\},$$
where the infimum is taken over all such representations of $z$. The reason for taking the completion is that $X\otimes Y$ endowed with the projective norm is complete if and only if either $X$ or $Y$ is finite dimensional (see \cite[p. 43, Exercises 2.4 and 2.5]{ryan}).

It is well known that $\|x \otimes y\|_{\pi} = \|x\| \|y\|$ for every $x \in X$, $y \in Y$, and that the closed unit ball of $X \pten Y$ is the closed convex hull of the set $B_X \otimes B_Y = \{ x \otimes y\colon  x \in B_X, y \in B_Y \}$. Throughout the paper, we will use both facts without any explicit reference.

Observe that every bilinear and continuous map  $G\in B(X\times Y)$ acts on $X \pten Y$ via
$$
G \left( \sum_{n=1}^{k} x_n \otimes y_n \right) = \sum_{n=1}^{k} G(x_n,y_n),$$
for $\sum_{n=1}^{k} x_n \otimes y_n \in X \otimes Y$. This action establishes a linear isometry from $B(X\times Y)$ onto $(X\pten Y)^*$ (see e.g.\ \cite[Theorem 2.9]{ryan}). Throughout this paper we will use the isometric identification $(X\pten Y)^*= B(X\times Y)$ without any explicit mention.

The question whether the Daugavet property is inherited by projective tensor products from its factors is one of the oldest open questions in the theory of Banach spaces with the Daugavet property which goes back, at least, to \cite[Section 6, (3)]{werner}. A relevant progress in the affirmative direction was given in \cite[Theorem 1.1]{rtv21}, where it was proved that $X\pten Y$ has the Daugavet property if both $X$ and $Y$ are $L_1$-preduals with the Daugavet property (this result was improved later in \cite[Theorem 5.4]{mr22}).

The key to proving the above mentioned \cite[Theorem 1.1]{rtv21} was to use the condition that the factors are $L_1$-preduals in order to take advantage of the fact that any compact operator taking values on an $L_1$-predual can be extended to any superspace in an almost norm-preserving way \cite[Theorem 6.1]{linds}. Because of this reason, we will get a version for the transfinite Daugavet property under the assumptions that each factor is an injective Banach space. Recall that an \emph{injective Banach space} is a Banach space $X$ with the property that if $T\colon Y\longrightarrow X$ is a bounded linear operator and $Y\subseteq Z$ then there exists a norm preserving extension of $T$ to $Z$, that is, a bounded linear operator $\hat T\colon Z\longrightarrow X$ with $\Vert\hat T\Vert=\Vert T\Vert$ and $\hat T(y)=T(y)$ for every $y\in Y$. We refer the reader to \cite[Section 4.3]{alka2006} for background around injective spaces.

Now the main result here is the following.

\begin{theorem}\label{theo:tensorprolargeDP}
Let $X$ and $Y$ be two injective Banach spaces and let $\kappa$ be an infinite cardinal. Then
\begin{enumerate}
    \item If $X$ and $Y$ have the $\kappa$-Daugavet property, then $X\pten Y$ has the $\kappa$-Daugavet property.
    \item If $X$ and $Y$ have the $\kappa$-perfect Daugavet property, then $X\pten Y$ has the $\kappa$-perfect Daugavet property.
\end{enumerate}
\end{theorem}

\begin{proof}
Let us prove (2), being the proof of (1) completely analogous. To do so, select $\{z_\alpha\colon  \alpha<\kappa\}\subseteq S_{X\pten Y}$ and a slice $S$ of $B_{X\pten Y}$, and let us find $x_0\in S_X$ and $y_0\in S_Y$ such that $x_0\otimes y_0\in S$ and
$$\Vert z_\alpha+x_0\otimes y_0\Vert=2$$
holds for every $\alpha<\kappa$. 

Indeed, write $S=S(B_{X\pten Y}, G,\delta)$ for certain norm-one bilinear map $G\colon X\times Y\longrightarrow \mathbb R$ and $\alpha>0$. Select $u_0\otimes v_0\in S$, which means $G(u_0,v_0)>1-\delta$. 
Given $\alpha<\kappa$ find a bilinear bounded map $B_\alpha$ with $\Vert B_\alpha\Vert=1$ and $B_\alpha(z_\alpha)=\Vert z_\alpha\Vert=1$.
In view of \cite[Proposition 2.8]{ryan} we can find, for every $\alpha<\kappa$, two bounded sequences $(x_n^\alpha)\subseteq X$ and $(y_n^\alpha)\subseteq Y$ such that
$$z_\alpha:=\sum_{n=1}^\infty x_n^\alpha\otimes y_n^\alpha.$$
Define $V:=\cspn\{x_n^\alpha\colon  \alpha<\kappa, n\in\mathbb N\}\subseteq X$ and $W:=\cspn\{y_n^\alpha\colon  \alpha<\kappa, n\in\mathbb N\}\subseteq Y$.

Now, the set $\{x\in B_X\colon  G(x,v_0)>1-\delta\}$ is a slice of $B_X$. Since $X$ has the $\kappa$-perfect Daugavet property, we can find by Proposition~\ref{lemma: subspace Daugavet} an element $x_0\in \{x\in B_X\colon  G(x,v_0)>1-\delta\}$ such that
$$\Vert v+\lambda x_0\Vert=\Vert v\Vert+\vert\lambda\vert$$
holds for every $v\in V$ and every $\lambda\in\mathbb R$. In a similar way, we can find $y_0\in S_Y$ such that $G(x_0,y_0)>1-\delta$ (that is, $x_0\otimes y_0\in S$) and 
$$\Vert w+\lambda y_0\Vert=\Vert w\Vert+\vert\lambda\vert$$
holds for every $w\in W$ and every $\lambda\in\mathbb R$.

We claim that, for every $\alpha<\kappa$, we have
$$\Vert z_\alpha+x_0\otimes y_0\Vert=2.$$
In order to prove it, select $\alpha<\kappa$ and $\varepsilon>0$. Since $\Vert B_\alpha\Vert=1$ we can find $a\in S_X$ and $b\in S_Y$ such that $B_\alpha(a,b)>1-\varepsilon$. Define $\phi\colon V\oplus\mathbb R x_0\longrightarrow X$ by the equation
$$\phi(v+\lambda x_0):=v+\lambda a.$$
We have that $\Vert \phi\Vert\leq 1$, just observe that given $v\in V$ and $\lambda\in\mathbb R$, 
\begin{align*}
\Vert \phi(v+\lambda x_0)\Vert =\Vert v+\lambda a\Vert\leq \Vert v\Vert+\vert \lambda\vert=\Vert v+\lambda x_0\Vert.
\end{align*}
By the injectivity of $X$, we can extend $\phi$ in a norm-preserving way to $\Phi\colon X\longrightarrow X$. In a similar way, we can construct an operator $\Psi\colon Y\longrightarrow Y$ such that $\Psi(w)=w$ holds for every $w\in W$ and $\Psi(y_0)=b$. 
Next, consider $\tilde B(x,y):=B_\alpha(\Phi(x),\Psi(y))$, which is a norm-one bilinear mapping since $\Phi$, $\Psi$ and $B_\alpha$ have norm one. Consequently, we get
\begin{align*}
\Vert z_\alpha+x_0\otimes y_0\Vert &\geq \tilde B(z_\alpha)+\tilde B(x_0\otimes y_0) = \tilde B_\alpha \left(\sum_{n=1}^\infty x_n^\alpha\otimes y_n^\alpha \right) +\tilde B(x_0\otimes y_0)\\
& =\sum_{n=1}^\infty B_\alpha(\Phi(x_n^\alpha),\Psi(y_n^\alpha))+B_\alpha(\Phi(x_0),\Psi(y_0)) >\sum_{n=1}^\infty B_\alpha(x_n^\alpha,y_n^\alpha)+1-\varepsilon\\
& =B_\alpha\left(\sum_{n=1}^\infty x_n^\alpha\otimes y_n^\alpha \right)+1-\varepsilon = B_\alpha(z_\alpha)+1-\varepsilon=2-\varepsilon.
\end{align*}
Since $\varepsilon$ was arbitrary, we conclude that $\Vert z_\alpha+x_0\otimes y_0\Vert=2$. 
\end{proof}

Some comments are pertinent:

\begin{remark}\label{remark:citainjekappa}
Examples of $\kappa$-perfect Daugavet spaces which are injective are exhibited in Remark~\ref{remark:C(K)injlargdaugavet}.
\end{remark}

\begin{remark}
 An inspection in the proof of Theorem~\ref{theo:tensorprolargeDP} reveals that, since we extend operators $\varphi$ defined on subspaces of density $\kappa$, it is enough to assume that $X$ and $Y$ are just $\kappa^+$ universally injective Banach spaces. Let us introduce necessary definition. Given a Banach space $X$ and an infinite cardinal $\aleph$, we say that $X$ is
\begin{enumerate}
\item \textit{$\aleph$ injective} if for any two normed spaces $Y\subseteq Z$ with $\dens(Z)< \aleph$ and any bounded linear operator $t\colon Y\longrightarrow X$ there exists a norm-preserving extension $T\colon Z\longrightarrow X$.
\item \textit{$\aleph$ universally injective} if for any two normed spaces $Y\subseteq Z$ with $\dens(Y)< \aleph$ and any bounded linear operator $t\colon Y\longrightarrow X$ there exists a norm-preserving extension $T\colon Z\longrightarrow X$.
\end{enumerate}
It is immediate that $\aleph$ universal injectivity implies $\aleph$ injectivity. The converse does hold true under GCH. We refer the interested reader to \cite{accgm15} for background.

Taking this into account, if $X$ and $Y$ are UD$_{<\kappa}$ spaces, then they enjoy the $\kappa'$-perfect Daugavet property for every $\kappa'<\kappa$ by Theorem~\ref{theo:universaldisdaugalarge}. Moreover, $X$ and $Y$ are $\kappa$ injective Banach spaces (c.f.\ \cite[Section 5.5.2]{sepibook} or \cite[Theorems 3.3 and 4.8]{mrz25}). Under GCH, we can conclude that $X$ and $Y$ are $\kappa$ universally injective Banach spaces \cite[Proposition 5.13]{sepibook}, and then $X\pten Y$ has the $\kappa'$-perfect Daugavet property.
\end{remark}

\section{Spaces of Lipschitz functions}\label{sect:Lipschitz}

A pointed metric space $(M, d, 0)$
is a metric space $(M, d)$ with a selected distinguished point $0$ in $M$, called the base point. For a pointed metric space $M$, we consider the Banach space $\Lip(M)$ formed by all real-valued Lipschitz functions that vanish at $0$, endowed with the norm given by the least Lipschitz constant, that is,
$$
\|f\|=\inf\{C>0\colon |f(x)-f(y)|\leq Cd(x,y)\;,\;\forall x,y\in M\}
$$
for every $f\in \Lip(M)$. 

We present now some definitions regarding metric spaces which are related to the fact that the spaces of Lipschitz functions have the Daugavet property. A metric space $(M,d)$ is a \textit{length space} for every $x,y\in M$, 
$d(x,y)$ is equal to the infimum of the lengths of the rectifiable curves joining $x$ and $y$ (in particular, $M$ is arc-connected).
Borrowing notation from \cite{ikw}, given $f\in S_{\Lip(M)}$ and $\varepsilon>0$, a point $m\in M$ is said to be an \textit{$\varepsilon$-point of $f$} if
$$\inf_{r>0}\|f_{|B(m,r)}\|>1-\ep.$$
Also, a metric space is said to be \textit{spreadingly local} whenever the set of $\varepsilon$-points of $f$ is infinite for every $\ep>0$ and every $f\in\Lip(M)$. Following the proof of Proposition 2.3 in \cite{ikw}, it is easy to see that every length space is spreadingly local, regardless it is complete or not. For complete metric spaces, the converse is also true \cite{gpr18}. 

With the above notation in mind, in \cite[Theorem 5.2]{cll23} it is proved that the following assertions are equivalent for a complete metric space $M$:
\begin{enumerate}
    \item $\Lip(M)$ has the Daugavet property,
    \item $\Lip(M)$ is $\omega$-rigid octahedral,
    \item $M$ is length.
\end{enumerate}
Let us remark that the equivalence between (1) and (3) already appeared in \cite[Theorem 3.5]{gpr18}.

The equivalence between (1) and (2) above, motivates us to get the following result.

\begin{theorem}\label{theo:Liperfe}
Let $M$ be a complete length metric space. Then $\Lip(M)$ has the $\omega$-perfect Daugavet property.
\end{theorem}

\begin{proof}
Fix $\{f_n\colon n\in\mathbb N\}\subseteq S_{\Lip(M)}$ and $g\in B_{\Lip(M)}$. Our aim is to find elements in the set $$\conv\bigl(\{f\in S_{\Lip(M)}\colon  \Vert f_n+f\Vert=2\ \forall n\in\mathbb N\}\bigr)$$ which are close to $g$. Let $p\in\mathbb N$ and choose $\varepsilon>0$ small enough to guarantee that $\varepsilon<\frac{1}{p}$.  Observe that, for every $n\in\mathbb N$ and any $k\in\mathbb N$, the set of $\frac{1}{k}$-points of $f_n$ is infinite. Take a bijection $\phi\colon \mathbb N^2\longrightarrow \mathbb N$. 

By an inductive argument, we can find a sequence of distinct points $\{t_{n,k}^i\colon  n\in\mathbb N, 1\leq i\leq p\}\setminus\{0\}$ such that each $t_{n,k}^i$ is a $\frac{1}{k}$-point. By an inductive argument again, we can find a family of pairwise disjoint open sets $\{W_{n,k}^i\colon  n,k\in\mathbb N,1\leq i\leq p\}$ with $t_{n,k}^i\in W_{n,k}^i$ and such that $0\notin W_{n,k}^i$ for every $n,k\in\mathbb N$, $1\leq i\leq p$. By the definition of $\varepsilon$-point, for every $n,k\in\mathbb N$ and $1\leq i\leq p$ we can find $r_{n,k}^i>0$ and a pair of distinct points $x_{n,k}^i, y_{n,k}^i$ with the following conditions:
\begin{enumerate}
    \item $B(x_{n,k}^i,r_{n,k}^i)\subseteq W_{n,k}^i$ (in particular the above balls are pairwise disjoint).
    \item $\frac{f_n(y_{n,k}^i)-f_n(x_{n,k}^i)}{d(x_{n,k}^i,y_{n,k}^i)}>1-\frac{1}{k}$ holds for every $n,k\in\mathbb N$ and every $1\leq i\leq p$,
    \item $2(1+\varepsilon)\frac{d(x_{n,k}^i,y_{n,k}^i)}{r_{n,k}^i-d(x_{n,k}^i,y_{n,k}^i)}<\varepsilon$ holds for every $n,k\in\mathbb N$ and every $1\leq i\leq p$ and,
    \item $2(1+\varepsilon)\frac{d(x_{m,j}^i,y_{m,j}^i)+d(x_{n,k}^i,y_{n,k}^i)}{r_{m,j}^i-d(x_{m,j}^i,y_{m,j}^i)}<\varepsilon$ holds for all $n,m,j,k\in\mathbb N$ such that $\phi(m,j)<\phi(n,k)$. 
\end{enumerate}
Now, given $1\leq i\leq p$, we define a function $g_i\colon M\setminus\bigcup\limits_{n,k\in\mathbb N} B(x_{n,k}^i,r_{n,k}^i)\cup \{x_{n,k}^i,y_{n,k}^i\}\longrightarrow \mathbb R$ by $g_i=g$ on $M\setminus\bigcup\limits_{n\in\mathbb N} B(x_n^i,r_n^i)$ and $g_i(x_{n,k}^i)=g(x_{n,k}^i)$. Define also $g_i(y_{n,k}^i)=g(x_{n,k}^i)+(1+2\varepsilon)d(x_{n,k}^i,y_{n,k}^i)$. 

Fix $1\leq i\leq p$ and let us estimate the Lipschitz norm of $g_i$. In order to do so, select $u\neq v\in M\setminus\bigcup\limits_{n\in\mathbb N} B(x_{n,k}^i,r_{n,k}^i)\cup \{x_{n,k}^i,y_{n,k}^i\}$, and let us estimate $A:=\frac{g_i(u)-g_i(v)}{d(u,v)}$. We have three possibilities:
\begin{enumerate}
    \item If $\{u,v\}\cap \{y_{n,k}^i\colon  n,k\in\mathbb N\}=\emptyset$ we have $g_i(u)=g(u)$ and $g_i(v)=g(v)$, from where $A\leq 1$.
    \item If $u=y_{n,k}^i$ for some $n\in\mathbb N$ and $v\neq y_{n,k}^i$ for every $n,k\in\mathbb N$ then $g_i(v)=g(v)$, whereas $g_i(y_{n,k}^i)=g(x_{n,k}^i)+(1+2\varepsilon)d(x_{n,k}^i,y_{n,k}^i)$. If $v=x_{n,k}^i$ then it is immediate that $A=1+2\varepsilon$. In other case we have
    \begin{align*}   
    A &=\frac{g(x_{n,k}^i)+(1+2\varepsilon)d(x_{n,k}^i,y_{n,k}^i)-g(v)}{d(y_{n,k}^i,v)} \leq \frac{d(x_{n,k}^i,v)+(1+2\varepsilon)d(x_{n,k}^i,y_{n,k}^i)}{d(y_{n,k}^i,v)}\\& \leq 1+(2+2\varepsilon)\frac{ d(x_{n,k}^i,y_{n,k}^i)}{r_{n,k}^i-d(x_{n,k}^i,y_{n,k}^i)} \leq 1+\varepsilon,
    \end{align*}
    where we have used that $d(v,x_n^i)\geq r_n^i$ as $v\notin B(x_n^i,r_n^i)$. 
    \item If $u=y_{n,k}^i$ and $v=y_{m,j}^i$ for $(n,k)\neq (m,j)$. In such case either $\phi(m,j)<\phi(n,k)$ or $\phi(n,k)<\phi(m,j)$. Assume, up to relabeling, that $\phi(m,j)<\phi(n,k)$. Then
    \begin{align*}   
        A& =\frac{g(x_{n,k}^i)+(1+2\varepsilon)d(x_{n,k}^i,y_{n,k}^i)-g(x_{m,j}^i)-(1+2\varepsilon)d(x_{m,j}^i,y_{m,j}^i)}{d(y_n^i,y_m^i)}\\
         &\leq \frac{d(x_{n,k}^i,x_{m,j}^i)+(1+2\varepsilon)(d(x_{n,k}^i,y_{n,k}^i)+d(x_{m,j}^i,y_{m,j}^i))}{d(y_{n,k}^i,y_{m,j}^i)}\\
        & =1+2(1+\varepsilon)\frac{d(x_{n,k}^i,y_{n,k}^i)+d(x_{m,j}^i,y_{m,j}^i)}{d(y_{n,k}^i,y_{m,j}^i)}\\
        & \leq 1+2(1+\varepsilon)\frac{d(x_{n,k}^i,y_{n,k}^i)+d(x_{m,j}^i,y_{m,j}^i)}{d(x_{m,j}^i,y_{n,j}^i)-d(x_{m,j}^i,y_{m,j}^i)}\\
        & \leq  1+2(1+\varepsilon)\frac{d(x_{n,k}^i,y_{n,k}^i)+d(x_{m,j}^i,y_{m,j}^i)}{r_{m,j}^i-d(x_{m,j}^i,y_{m,j}^i)}\leq 1+\varepsilon.
    \end{align*}
\end{enumerate}
It follows that the Lipschitz constant of $g_i$ is $1+2\varepsilon$. Now,  McShane extension theorem \cite[Theorem 1.33]{weaver18} provides an extension of $g_i$ to the whole $M$, which we still call $g_i$. Observe that $\supp (g_i-g)\subseteq \bigcup\limits_{n\in\mathbb N} B(x_{n,k}^i,r_{n,k}^i)$. Write $h_i:=\frac{g_i}{(1+2\varepsilon)}\in B_{\Lip(M)}$. Given $n\in\mathbb N$, let us prove that $\Vert f_n+h_i\Vert=2$. Indeed, for every $k\in\mathbb N$ we have that
$$\Vert f_n+h_i\Vert\geq \frac{f_n(y_{n,k}^i)-f_n(x_{n,k}^i)}{d(y_n^i,x_n^i)}+\frac{1}{1+2\varepsilon}\frac{g_i(y_{n,k}^i)-g_i(x_{n,k}^i)}{d(y_n^i,x_n^i)}>1-\frac{1}{k} +1=2-\frac{1}{k}.$$
Hence, $\Vert f_n+f\Vert=2$ by the arbitrariness of $k\in \N$.
It remains to prove that $\frac{1}{p}\sum_{i=1}^p h_i$ is close to $g$ in norm. In order to do so, select $u,v\in M, u\neq v$. We have
$$\left\vert \frac{g(u)-\frac{1}{p}g_i(u)-g(v)-\frac{1}{p}\sum_{i=1}^p g_i(v)}{d(u,v)}\right\vert =\frac{1}{p}\sum_{i=1}^p\frac{\vert g(u)-g_i(u)-( g(v)-g_i(v))\vert}{d(u,v)}.$$
Since the sets $\supp(g-g_i)$ are pairwise disjoint, there is at most one index $1\leq i_u\leq p$ such that $g(u)-g_{i_u}(u)\neq 0$. Similarly, there is at most one index $1\leq i_v\leq p$ such that $g(u)-g_{i_v}(v)\neq 0$. Finally, given $i\in\{i_u,i_v\}$, we have
$$\vert g(u)-g_i(u)-(g(v)-g_i(v))\vert\leq \vert g(u)-g(v)\vert +\vert g_i(u)-g_i(v)\vert\leq 2d(u,v).$$
Consequently, 
$$\frac{1}{p}\sum_{i=1}^p\frac{\vert g(u)-g_i(u)-( g(v)-g_i(v))\vert}{d(u,v)}\leq \frac{2}{p}\frac{2d(u,v)}{d(u,v)}=\frac{4}{p}.$$
Since $u,v\in M$ with $u\neq v$ were arbitrary, we infer that  
$$\left\Vert g-\frac{1}{p}\sum_{i=1}^p g_i \right\Vert\leq \frac{4}{p}.$$
Now, since 
$$\Vert h_i-g_i\Vert=\left\vert \frac{1}{1+2\varepsilon}-1\right\vert \Vert g_i\Vert=\frac{\Vert g_i\Vert}{1+\varepsilon}2\varepsilon\leq 2\varepsilon,$$
we deduce that
$$\left\Vert g-\frac{1}{p}\sum_{i=1}^p h_i \right\Vert\leq \left\Vert g-\frac{1}{p}\sum_{i=1}^p g_i \right\Vert+\frac{1}{p}\sum_{i=1}^p \Vert g_i-h_i\Vert <\frac{4}{p}+2\varepsilon<\frac{6}{p}.$$
The arbitrariness of $p\in \N$ and $g\in B_{\Lip(M)}$ implies that $$B_{\Lip(M)}=\overline{\conv}\bigl( \{ f\in B_{\Lip(M)}\colon  \Vert f_n+f\Vert=2\  \forall n\in\mathbb N\}\bigr),$$ from where we conclude that $\Lip(M)$ has the $\omega$-perfect Daugavet property by Proposition~\ref{prop:charlargeDPconvexhull}.
\end{proof}

The above theorem allows to obtain an improvement of \cite[Theorem 5.1]{cll23} and \cite[Theorem 3.5]{gpr18} in the following lines.

\begin{corollary}\label{coro:Lipcharlength}
Let $M$ be a complete metric space. Then, the following assertions are equivalent:
\begin{enumerate}
    \item $\Lip(M)$ has the $\omega$-perfect Daugavet property.
    \item $\Lip(M)$ has the Daugavet property.
    \item $\Lip(M)$ is $\omega$-rigid octahedral,
    \item $M$ is length.
\end{enumerate}
\end{corollary}

Let us comment that it has been recently shown in \cite{KAASIK2026} that for complete length metric spaces, the Lipschitz free space $\mathcal F(M)$ over $M$, has the ASD2P, that is, for every convex combination of slices of the unit ball $C$ there are $x,y\in C$ such that $\|x-y\|=2$. As far as we know, this result is independent from our Theorem~\ref{theo:Liperfe}.

\section{Remarks and open questions}\label{sect:remarkopeque}

In view of Proposition~\ref{prop:necelinfisum} and the result for the $C(K)$-spaces case (Proposition~\ref{prop:linftysum-ckspaces}), it would be interesting to know if the converse holds true.

\begin{question}\label{ques:linftysum}
Let $\{X_i\colon  i\in I\}$ be an arbitrary family of Banach spaces, is it true that
$$\mathfrak{Dau}\left(\left(\bigoplus\nolimits_{i\in I} X_i\right)_{\ell_\infty}\right) = \min\limits_{i\in I}\{\mathfrak{Dau}(X_i)\}\quad \text{and} \quad \mathfrak{pDau}\left(\left(\bigoplus\nolimits_{i\in I} X_i\right)_{\ell_\infty}\right) = \min\limits_{i\in I}\{\mathfrak{pDau}(X_i)\}?$$
\end{question}

Concerning the result of Section~\ref{sect:Lipschitz}, it is clear that in Theorem~\ref{theo:Liperfe} we can not improve such theorem for general metric spaces. Indeed $\Lip([0,1])$, being isometrically isomorphic to $L_\infty([0,1])$, has density $c$. Because of that reason we wonder.

\begin{question}
Let $M$ be a length metric space and let $\kappa$ be an uncountable cardinal. Can the fact that $\mathfrak{Dau}\bigl(\Lip(M)\big)=\kappa$ be characterised in terms of a geometric property on $M$? Are $\mathfrak{Dau}\bigl(\Lip(M)\big)$ and $\mathfrak{pDau}\bigl(\Lip(M)\big)$ always equal?
\end{question}

For every Banach space $X$, the failure of the Daugavet property for $X$ is equivalent to $\mathfrak{Dau}(X)=1$, but it is also equivalent to $\mathfrak{Dau}(X)<\omega$ by \cite[Lemma 3.1.14]{kmrzw25}. This index never takes finite values different from 1. We do not know if the same holds for the perfect Daugavet property. 

\begin{question}
    If a Banach space $X$ satisfies that $\mathfrak{pDau}(X)>1$, does it follow that $\mathfrak{pDau}(X)\geq \omega$?
\end{question}

\begin{question}
    Is there a Banach space \emph{isometric} to a dual space for which $\mathfrak{Dau}(X)=\omega$?
\end{question}
We know that isomorphic to a dual space is possible, see Proposition~\ref{dualremark}. Next, we summarize some the questions about the reaping number of compact spaces.

\begin{question}\label{questionminimalpicharacter} Does the inequality $\min\{\pi_\chi(x) \colon  x\in K \} \leq \mathfrak{r}(K)^+$ hold for every compact space $K$?
\end{question}

\begin{question}
    Is $\ddot{\mathfrak{r}}_{n,k}(K) = \tilde{\mathfrak{r}}_{n,k}(K)$ for all compact spaces $K$ and all $2\leq k\leq n$?
\end{question}

\begin{question} For compact spaces $K$ and $L$, is $\mathfrak{r}(K\times L) = \mathfrak{r}(K)\vee \mathfrak{r}(L)$? Do we have, at least, that $\mathfrak{r}(K\times L) \leq \mathfrak{r}(K)^+\vee \mathfrak{r}(L)^+$
\end{question}

We may wonder if there is a way to reduce the study of the reaping number in general compact spaces to the better behaved class of totally disconnected compact spaces.

\begin{question}\label{irreduciblequestion}
Given a compact space $K$, does there exist a totally disconnected compact space $L$ and an irreducible surjection $L\longrightarrow K$ such that $\mathfrak{r}(L)=\mathfrak{r}(K)$?
\end{question} 

A positive answer to Question~\ref{irreduciblequestion} would imply a positive answer to Question~\ref{questionminimalpicharacter}, because we would have $\tilde{\mathfrak{r}}_\omega(K) \leq \tilde{\mathfrak{r}}_\omega(L)\leq {\mathfrak{r}}(L)^+ = \mathfrak{r}(K)^+$, where the first inequality is an easy property of irreducible surjections similar to Proposition~\ref{irreduciblereduction} and the second inequality is Theorem~\ref{theo:dsw}.

\section*{Acknowledgements}

The authors thank Tommaso Russo for sharing the proof of \cite[Proposition 6.11]{brss26}, which motivated the proof of Theorem~\ref{theo: L1 has k-Daugavet}. They also thanks an anonymous referee for pointing out that Proposition~\ref{prop:sufic0sum} contained a gap in the first version of the manuscript and for providing Remark~\ref{remark:c0sumfalso}. We also acknowledge the use of ChatGPT, Claude, and Microsoft Copilot 365 for polishing the final version of this manuscript. 

The work of A.\ Avil\'es, J.\ Langemets, M.\ Mart\'in, and A.\ Rueda Zoca was supported by:
\begin{itemize}
    \item MICIU/AEI/10.13039/501100011033 and ERDF/EU through the grants PID2021-122126NB-C31 (Mart\'in and Rueda Zoca) and  PID2021-122126NB-C32 (Avil\'es).
    \item Estonian Research Council grant PRG2545 (Langemets).
    \item ``Maria de Maeztu'' Excellence Unit IMAG, funded by MICIU/AEI/10.13039/501100011033 with reference CEX2020-001105-M (Mart\'in and Rueda Zoca).
    \item Fundaci\'on S\'eneca, ACyT Regi\'on de Murcia, grant 21955/PI/22 (Avil\'es and Rueda Zoca).
    \item Junta de Andaluc\'ia, grant FQM-0185 (Mart\'in and Rueda Zoca).
\end{itemize}

\end{document}